\documentclass[11pt,oneside,reqno]{article}

\usepackage[a4paper,width=15cm,height=24cm]{geometry}

\usepackage{tikz}

\usepackage{array,booktabs,longtable,caption}

\usepackage{amsmath,amsthm,amssymb,mathrsfs,bbm}

\usepackage[authoryear,longnamesfirst]{natbib}
\renewcommand{\cite}{\citet*}

\usepackage[breaklinks=true]{hyperref}

\theoremstyle{plain}
\newtheorem{theorem}{Theorem}[section]

\newtheorem{lemma}[theorem]{Lemma}
\newtheorem{corollary}[theorem]{Corollary}

\theoremstyle{definition}

\newtheorem{remark}[theorem]{Remark}

\makeatletter

\def\be#1{\begin{equation*}#1\end{equation*}}
\def\ben#1{\begin{equation}#1\end{equation}}
\def\bes#1{\begin{equation*}\begin{split}#1\end{split}\end{equation*}}
\def\besn#1{\begin{equation}\begin{split}#1\end{split}\end{equation}}
\def\bg#1{\begin{gather*}#1\end{gather*}}

\def\bm#1{\begin{multline*}#1\end{multline*}}

\def\ba#1{\begin{align*}#1\end{align*}}

\def\given{\typeout{Command 'given' should only be used within bracket command}}
\newcounter{@bracketlevel}
\def\@bracketfactory#1#2#3#4#5#6{
\expandafter\def\csname#1\endcsname##1{%
\addtocounter{@bracketlevel}{1}%
\global\expandafter\let\csname @middummy\alph{@bracketlevel}\endcsname\given%
\global\def\given{\mskip#5\csname#4\endcsname\vert\mskip#6}\csname#4l\endcsname#2##1\csname#4r\endcsname#3%
\global\expandafter\let\expandafter\given\csname @middummy\alph{@bracketlevel}\endcsname
\addtocounter{@bracketlevel}{-1}}%
}
\def\bracketfactory#1#2#3{%
\@bracketfactory{#1}{#2}{#3}{relax}{1mu plus 0.25mu minus 0.25mu}{0.6mu plus 0.15mu minus 0.15mu}
\@bracketfactory{b#1}{#2}{#3}{big}{1mu plus 0.25mu minus 0.25mu}{0.6mu plus 0.15mu minus 0.15mu}
\@bracketfactory{bb#1}{#2}{#3}{Big}{2.4mu plus 0.8mu minus 0.8mu}{1.8mu plus 0.6mu minus 0.6mu}
\@bracketfactory{bbb#1}{#2}{#3}{bigg}{3.2mu plus 1mu minus 1mu}{2.4mu plus 0.75mu minus 0.75mu}
\@bracketfactory{bbbb#1}{#2}{#3}{Bigg}{4mu plus 1mu minus 1mu}{3mu plus 0.75mu minus 0.75mu}
}
\bracketfactory{clc}{\lbrace}{\rbrace}
\bracketfactory{clr}{(}{)}
\bracketfactory{cls}{[}{]}
\bracketfactory{abs}{\lvert}{\rvert}
\bracketfactory{norm}{\Vert}{\Vert}
\bracketfactory{floor}{\lfloor}{\rfloor}
\bracketfactory{ceil}{\lceil}{\rceil}
\bracketfactory{angle}{\langle}{\rangle}

\newcounter{ctr}\loop\stepcounter{ctr}\edef\X{\@Alph\c@ctr}%
	\expandafter\edef\csname s\X\endcsname{\noexpand\mathscr{\X}}
	\expandafter\edef\csname c\X\endcsname{\noexpand\mathcal{\X}}
	\expandafter\edef\csname b\X\endcsname{\noexpand\boldsymbol{\X}}
	\expandafter\edef\csname I\X\endcsname{\noexpand\mathbbm{\X}}
\ifnum\thectr<26\repeat

\newcount\minute
\newcount\hour
\newcount\hourMins
\def\now{%
\minute=\time%
\hour=\time \divide \hour by 60%
\hourMins=\hour \multiply\hourMins by 60%
\advance\minute by -\hourMins%
\zeroPadTwo{\the\hour}:\zeroPadTwo{\the\minute}%
}
\def\zeroPadTwo#1{\ifnum #1<10 0\fi#1}

\numberwithin{equation}{section}
\allowdisplaybreaks[4]

\renewcommand\section{\@startsection {section}{1}{\z@}%
{-3.5ex \@plus -1ex \@minus -.2ex}%
{1.3ex \@plus.2ex}%
{\center\small\sc\mathversion{bold}\MakeUppercase}}

\def\subsection#1{\@startsection {subsection}{2}{0pt}%
{-3.5ex \@plus -1ex \@minus -.2ex}%
{1ex \@plus.2ex}%
{\bf\mathversion{bold}}{#1}}

\def\subsubsection#1{\@startsection{subsubsection}{3}{0pt}%
{\medskipamount}%
{-10pt}%
{\normalsize\itshape}{\kern-2.2ex. #1.}}

\def\blfootnote{\xdef\@thefnmark{}\@footnotetext}

\makeatother

\def\^#1{\relax\ifmmode{\mathaccent"705E #1}\else{\accent94 #1}\fi}
\def\~#1{\relax\ifmmode{\mathaccent"707E #1}\else{\accent"7E #1}\fi}

\edef\-#1{\noexpand\ifmmode {\noexpand\bar{#1}} \noexpand\else \-#1\noexpand\fi}
\def\>#1{\vec{#1}}
\def\.#1{\dot{#1}}

\def\atop{\@@atop}

\renewcommand{\leq}{\leqslant}
\renewcommand{\geq}{\geqslant}
\renewcommand{\phi}{\varphi}

\newcommand{\eq}{\eqref}

\newcommand{\dw}{\mathop{d_{\mathrm{W}}}}
\newcommand{\dk}{\mathop{d_{\mathrm{K}}}}

\newcommand{\toinf}{\to\infty}

\newcommand{\Var}{\mathop{\mathrm{Var}}\nolimits}
\newcommand{\Cov}{\mathop{\mathrm{Cov}}\nolimits}
\newcommand{\law}{\mathscr{L}}
\newcommand{\I}{\mathrm{I}}

\newcommand{\Be}{\mathrm{Be}}
\newcommand{\ahalf}{{\textstyle\frac12}}
\newcommand{\tssum}{\textstyle\sum}
\def\ds{\displaystyle}
\def\ER{Erd\H{o}s-R\'enyi}

\begin{document}

\title{\sc\bf\large\MakeUppercase{Kolmogorov bounds for the normal approximation of the number of triangles in the Erd\H{o}s-R\'enyi random graph}}
\author{\sc Adrian R\"ollin}
\date{\it National University of Singapore}
\maketitle

\begin{abstract} We bound the error for the normal approximation of the number of triangles in the \ER\ random graph with respect to the Kolmogorov metric. Our bounds match the best available Wasserstein-bounds obtained by \cite{Barbour1989}, resolving a long-standing open problem. The proofs are based on a new variant of the Stein-Tikhomirov method --- a combination of Stein's method and characteristic functions introduced by \cite{Tikhomirov1980}. 
\end{abstract}

\section{Introduction}

Consider the \ER\ random graph~$G(n,p)$, where~$p$ is allowed to depend on~$n$. The conditions under which subgraph counts exhibit asymptotic normality were fully characterised by \cite{Rucinski1988}; in the particular case of triangles, a normal limit holds if and only if 
\ben{\label{1a}
	\lim_{n\toinf} np = \infty
	\qquad\text{and}\qquad
	\lim_{n\toinf} n^2(1-p) = \infty.
}
This result was complemented by corresponding rates of convergence by \cite{Barbour1989} for the Wasserstein metric~$\dw$. Let~$T$ denote the number of triangles in~$G(n,p)$, and let~$W=(T-\IE T)/\sqrt{\Var T}$ denotes the centred and normalised number of triangles; \cite{Barbour1989} showed that there is a universal constant~$C$ such that
\ben{\label{1}
	\dw\clr{\law(W),\Phi}\leq \begin{cases}
			\ds\frac{C}{n(1-p)^{1/2}}	
				& \text{if~$1/2<p<1$,}\\[2ex]
			\ds\frac{C}{np^{1/2}}
				& \text{if~$n^{-1/2}< p \leq 1/2$,}\\[2ex]
			\ds\frac{C}{n^{3/2}p^{3/2}}
				& \text{if~$0<p\leq n^{-1/2}$,}
	\end{cases}
}
where~$\Phi$ is the standard normal distribution function.
This bound indeed goes to zero exactly under the conditions of asymptotic normality obtained by \cite{Rucinski1988}.

Obtaining bounds with respect to the Kolmogorov metric~$d_K$ has turned out to be much harder, with no progress until recently. A straightforward bound, as is well known, can be obtained through the relation
\ben{\label{2}
	\dk(\law(W),\Phi)\leq \sqrt{\dw(\law(W),\Phi)}
} 
(see e.g.\ \cite[Proposition~1.2]{Ross2011}).
To the best of our knowledge, the first result that improves on~\eq{2} was obtained by \cite[Theorem~4.11]{Rollin2012a}, which is a stronger total variation approximation by a translated Poisson distribution. Better rates for the Kolmogorov metric were obtained later by \cite{Krokowski2015} using Stein's method in combination with Malliavin-type methods.

The following is our main result.

\begin{theorem}\label{THM1} Let~$W$ be the centred and normalised number of triangles in~$G(n,p)$. There is a universal constant~$C$ such that, for every~$n\geq 3$ and every~$0<p<1$,
\ben{\label{3}
	\dk\bclr{\law(W),\Phi}
	\leq \begin{cases}
			\ds\frac{C}{n(1-p)^{1/2}}	
				& \text{if~$1/2<p<1$,}\\[2ex]
			\ds\frac{C}{np^{1/2}}
				& \text{if~$n^{-1/2}< p \leq 1/2$,}\\[2ex]
			\ds\frac{C}{n^{3/2}p^{3/2}}
				& \text{if~$0<p\leq n^{-1/2}$.}
	\end{cases}
}
\end{theorem}

\begin{remark} 
In order to analyse the optimality of the bound \eq{3}, consider the variance of the number of triangles (see Lemma~\ref{lem7}), which satisfies 
\be{
\Var T \asymp \begin{cases}
			n^4(1-p)	
				& \text{if~$1/2<p<1$,}\\[2ex]
			n^4p^5
				& \text{if~$n^{-1/2}< p \leq 1/2$,}\\[2ex]
			n^3p^3
				& \text{if~$0<p\leq n^{-1/2}$,}
	\end{cases}
}
where we write~$f(n)\asymp g(n)$ if the ratio~$f(n)/g(n)$ is bounded away from 0 and infinity as~$n\toinf$. 
If~$0<p\leq n^{-1/2}$, the bound~\eq{3} is of order~$1/\sqrt{\Var T}$, which is best possible for an integer-valued random variable normalised by its standard deviation. In the other two cases, the dependence between the triangle indicators becomes so strong that the covariance terms change the order of the variance, but the following example shows that the rate in~\eq{3} is still within what can be expected for sums of random variables with a similar covariance structure.

For~$1\leq i<j<k\leq n$, let~$I_{ij} \sim \Be(p)$ and~$I_{ijk} \sim \Be(p^2)$ be independent random variables, and let~$X_{ijk}=I_{ij}I_{ijk}$. With~$Y = \sum_{1<i<j<k\leq n} X_{ijk}$ we have $\IE Y = {n\choose 3}p^3 = \IE T$ and 
\be{
	s^2:=\Var Y 
	= {n\choose 3}\bclr{p^3(1-p)^3+(n-3)p^5(1-p)}
  \asymp \Var T.
}
 Since~$Y$ is the sum of~${n\choose 2}$ independent and identically distributed random variables, each being distributed like~$I_{12}\sum_{3\leq k\leq n} I_{12k}$, we can apply the Berry-Esseen theorem and obtain
\ben{\label{4}
	\dk\bclr{(Y-\IE Y)/s,\Phi}\leq \frac{Cn^2\gamma}{s^3},
	\qquad\text{where}
	\quad
	\gamma = \IE\bbabs{I_{12}\sum_{3\leq k\leq n}I_{12k}-(n-2)p^3}^3.
}
Now,
\be{
	\gamma
	 = p\IE\bbabs{\sum_{3\leq k\leq n}(I_{12k}-p^2) + (n-2)p^2(1-p)}^3 + (1-p)(n-2)^3p^9.
}
First, consider the case where $n^{-1/2}<p\leq 1/2$; we can use Bernstein's inequality to conclude that $\sum_{3\leq k\leq n}(I_{12k}-p^2)$ is strongly concentrated around $0$, so that $\gamma\asymp n^3p^7$. Recalling that $s^2\asymp n^4p^5$, it follows that the bound in \eq{4} is of order $n^{-1}p^{-1/2}$, which is the same as that of~\eq{3}. 
In the case where $1/2<p<1$, we can bound
\besn{\label{4b}
	&\IE\bbabs{\sum_{3\leq k\leq n}(I_{12k}-p^2) + 	(n-2)p^2(1-p)}^3 \\
	&\qquad\leq  \IE\bbabs{\sum_{3\leq k\leq n}(I_{12k}-p^2)}^3 + (n-2)^3p^6(1-p)^3 \\
	&\qquad\leq  C\bclr{\max\bclc{n(1-p),n^2(1-p)^2}}^{3/4} + Cn^3(1-p)^3.
}
If~$\lim n(1-p)=\infty$, the bound~\eq{4b} is of order~$n^3(1-p)^3$, and so~$\gamma\asymp n^3(1-p)$. If~$\lim \sup n(1-p)<\infty$, we also have $\gamma\asymp n^3(1-p)$, as~$\lim n^2(1-p)=\infty$ by \eq{1a}. Recalling that $s^2\asymp n^4(1-p)$, we conclude that the bound in \eq{4} is of order~$n^{-1}(1-p)^{-1/2}$, which, again, is the same as that of~\eq{3}. 

Since the Berry-Esseen theorem gives optimal rates in general, this example strongly indicates that~\eq{3} indeed yields the correct rate of convergence also for~$n^{-1/2}<p\leq 1/2$ and $1/2<p<1$.
\end{remark}

\section{The Stein-Tikhomirov method}

Our proof is based on a method introduced by \cite{Tikhomirov1980}, who takes elements from Stein's method, initiated by \cite{Stein1972}, and combines them with characteristic functions. As far as we can tell, the Stein-Tikhomirov method has only been successfully applied to prove CLTs for sums of random variables with temporal or spatial dependence and some mixing conditions; see for example \cite{Bulinskii1996}, who obtained CLTs for associated random variables index by the~$d$-dimensional lattice assuming an exponential decay of the covariances. In order to apply the Stein-Tikhomirov method to triangle counts, we develop in this section a new abstract theorem by combining Tikhomirov's approach with ideas from the more recent literature around Stein's method; we use, in particular, \emph{Stein couplings} from \cite{Chen2010b}.

To this end, we say a triple of random variables~$(W,W',G)$ is a Stein coupling if
\ben{\label{5}
	\IE\clc{Gf(W')-Gf(W)} = \IE\clc{Wf(W)}
}
for all functions~$f$ for which the expectations exist. We refer to \cite{Chen2010b} for examples and applications of Stein couplings.

Although the random variables of interest are real valued, we will need to consider complex valued random variables due to the use of characteristic functions. If~$z=x+iy\in\IC$, we denote by~$z^*=x-iy$ its complex conjugate. If~$X=X_1+iX_2$ is a complex random variables, we define as usual
\be{
	\Var X = \IE\clc{(X-\IE X)(X-\IE X)^*},
}
and if~$Y$ is another complex random variable, we define
\be{
	\Cov(X,Y)=\IE\clc{(X-\IE X)(Y-\IE Y)^*}
}
All the usual properties of the variance and covariance functions remain the same, except that~$\Cov(X,Y)=\Cov(Y,X)^*$, from which we conclude that~$\Var(aX)=\abs{a}^2\Var(X)$ and~$\Cov(aX,bY)=ab^*\Cov(X,Y)$ for any~$a,b\in\IC$.

The following is our main result in this section. 

\begin{theorem}\label{THM2} Let~$(W,W',G)$ be a Stein coupling with~$\Var W = 1$. Then
\ben{\label{6}
	\sup_{x\in \IR}\abs{\IP[W\leq x]-\IP[Z\leq x]}
	\leq 0.38r_1+3.05\~r_1+ 0.64r_2\bbclr{1+2\log_+\frac{1}{2\~r_1}},
}
where, with~$D=W'-W$,
\be{
	r_1=\IE\abs{GD^2},
	\qquad
	r_2 = \sup_{t\neq0}\frac{\bcls{\Var\IE^W(G(e^{itD}-1))}^{1/2}}{\abs{t}},
}	
and where~$\~r_1\geq r_1$ is arbitrary.
If, in addition, there are random variables~$\~D$, $S$ and~$W''$ such that~$\IE^W(G\~D)=\IE^W(GD)$, such that~$\IE^W S = 1$,  and such that~$\IE^{W''}(G\~D)=\IE^{W''}S$, then
\ben{\label{7}
	\sup_{x\in \IR}\abs{\IP[W\leq x]-\IP[Z\leq x]}
	\leq 0.76r_3+ 6.10\~r_3+ 0.64\frac{r_4}{\~r_3},
}
where, with~$D'=W''-W$,
\ba{
	r_3 & = \ahalf\IE\abs{GD^2} + \IE\abs{G\~DD'}+\IE\abs{SD'},\\
	r_4 & = \sup_{t\neq0}\frac{\bcls{\Var\IE^W(G(e^{itD}-1-itD))}^{1/2}}{t^2}
	 + \sup_{t\neq0}\frac{\bcls{\Var\IE^W(G\~D(e^{itD'}-1))}^{1/2}}{\abs{t}} \\ 
	&\qquad+ \sup_{t\neq0}\frac{\bcls{\Var\IE^W(S(e^{itD'}-1))}^{1/2}}{\abs{t}},
}
and where~$\~r_3\geq r_3$ is arbitrary.
\end{theorem}

Although the terms~$r_1$ and~$r_2$ in~\eq{6} are much simpler
than the terms~$r_3$ and~$r_4$ in~\eq{7}, the bound~\eq{6} comes at the cost of an additional logarithmic term, so that in order to prove Theorem~\ref{THM1}, we will have to make use of~\eq{7}.

To prove Theorem~\ref{THM2}, we will use the following, classic smoothing lemma due to Esseen; see, for example, \cite[Theorem~1.5.2, p.~27/28]{Ibragimov1971}.

\begin{lemma}\label{lem1} Let~$W$ be a random variable with characteristic function~$\phi(t)=\IE e^{itW}$, and let~$Z$ have a standard normal distribution. Then, for any~$T\nolinebreak>\nolinebreak0$,
\be{
	\sup_{x\in \IR}\abs{\IP[W\leq x]-\IP[Z\leq x]}
	\leq \frac{1}{\pi}\int_{-T}^T \frac{\abs{\phi(t)-e^{-t^2/2}}}{\abs{t}}dt + \frac{24}{\pi\sqrt{2\pi}T}.
}
\end{lemma}

The next lemma is an adapted and more explicit version of what is used implicitly by \cite{Tikhomirov1980} (see in particular  Equation (3.22) therein)

\begin{lemma}\label{lem2} Let~$W$ be an integrable random variable, and assume its characteristic function~$\phi(t)=\IE e^{itW}$ satisfies the differential equation
\be{
	\phi'(t) = -t(1+a(t))\phi(t)+b(t),\qquad t\in\IR,
}
where~$a(t)$ and~$b(t)$ are (possibly complex valued) functions satisfying
\be{
	\abs{a(t)}\leq A_0+ A_1\abs{t},
	\qquad
	\abs{b(t)}\leq B_0 + B_1\abs{t} + B_2t^2
}
for some non-negative constants~$A_0<1/2$, $A_1$, $B_0$, $B_1$ and~$B_2$. Then, for any~$t\geq 2A_1/(1-2A_0)$,
\bes{\label{8}
	&\sup_{x\in \IR}\abs{\IP[W\leq x]-\IP[Z\leq x]}\\
	&\qquad\leq {\frac{2}{\pi}A_0+\frac{4}{3\sqrt{\pi}}A_1 
	+ \frac{\sqrt{\pi}}{2}B_0 
	+\frac{2}{\pi}B_1\bbclr{1+2\log_+\frac{1}{2t}}
	+ \frac{4}{\pi}\frac{B_2}{t} + \frac{24t}{\pi\sqrt{2\pi}}}.
}
\end{lemma}
\begin{proof} The ODE
\be{
	\phi'(t) = \tilde{a}\phi(t)+b(t),\qquad t\in\IR,
}
with boundary condition~$\phi(0)=1$ has solution
\bes{
	\phi(t) 
	& = {\exp\bbclr{\int_0^t \~a(u)du}+\int_0^t \exp\bbclr{\int_u^t \~a(v)dv}b(u)du}.
}
Hence, for~$\~a(t)=-t-ta(t)$ and with
\be{
	\Delta(t) = \int_0^t ua(u)du,
}
we have
\bes{
	\phi(t) &= \exp\bbclr{-\frac{t^2}{2}-\Delta(t)}
						+ \exp\bbclr{-\frac{t^2}{2}-\Delta(t)}\int_0^t \exp\bbclr{\frac{u^2}{2}+\Delta(u)}b(u)du \\
	& = \exp\bbclr{-\frac{t^2}2} + R_1(t) + R_2(t),
}
where
\ba{
	R_1(t) & =\exp\bbclr{-\frac{t^2}{2}-\Delta(t)}-\exp\bbclr{-\frac{t^2}{2}},\\
	R_2(t)&={\exp\bbclr{-\frac{t^2}{2}-\Delta(t)}\int_0^t \exp\bbclr{\frac{u^2}{2}+\Delta(u)}b(u)du}.
}
Using 
\ben{\label{9}
	e^{x+h} -  e^x= h\int_0^1 e^{x+hs}ds,
}
we write
\be{
	R_1(t) = -\Delta(t)\int_{0}^1 \exp\bbclr{-\frac{t^2}{2}-s\Delta(t)}ds.
}
Using the bound on~$a(t)$,
\be{
	\abs{\Delta(t)}
	\leq \bbbabs{\int_0^t u a(u) du} 
	\leq \frac{A_0}{2}t^2+\frac{A_1}{3}\abs{t}^3 .
}
Since
\bg{
\frac{A_0}{2}+\frac{A_1}{3}\abs{t}\leq \frac{1}{4}
\quad\iff\quad 
\abs{t}\leq \frac{3(1-2A_0)}{4A_1} =: T_1,
}
we have
\be{
	-\frac{t^2}{2} - s \Delta(t) \leq -\frac{t^2}{4} \qquad\text{for~$\abs{t}\leq T_1$}
}
for all~$0\leq s\leq 1$. This yields 
\be{
	\abs{R_1(t)}=\bbbabs{\Delta(t)\int_{0}^1 \exp\bbclr{-\frac{t^2}{2}-s\Delta(t)}ds} 
	\leq \bbclr{\frac{A_0}{2}t^2+\frac{A_1}{3}\abs{t}^3}\exp\bbclr{-\frac{t^2}{4}}\quad\text{for~$\abs{t}\leq T_1$}.
}
In order to bound~$R_2(t)$, write
\be{
	R_2(t) = {\int_0^t \exp\bbclr{-\bbclr{\frac{t^2}{2}-\frac{u^2}{2}}-(\Delta(t)-\Delta(u))}b(u)du}.
}
Whenever~$0\leq u\leq t$ or~$t\leq u\leq 0$, we have
\be{
	\babs{\Delta(t)-\Delta(u)}
		=\bbbabs{\int_u^t va(v)dv}
		\leq (A_0+A_1\abs{t})\bbbabs{\int_u^t vdv}
		\leq \frac{A_0+A_1\abs{t}}{2}(t^2-u^2) .
}
Thus, since
\bg{
	\frac{A_0+A_1\abs{t}}{2}\leq \frac{1}{4} 
	\quad\iff\quad
	\abs{t}\leq \frac{1-2A_0}{2A_1} =: T_2 ,
}
we obtain 
\be{
\abs{R_2(t)} 
	 \leq \exp\bbclr{-\frac{t^2}{4}}\bbbabs{\int_0^t \abs{b(u)}\exp\bbclr{\frac{u^2}{4}}du}\qquad\text{for~$\abs{t}\leq T_2$.}
}
Applying the bound on~$\abs{b(t)}$, this yields 
\bes{
	\abs{R_2(t)} 
	& \leq B_0\exp\bbclr{-\frac{t^2}{4}}\int_0^{\abs{t}}\exp\bbclr{\frac{u^2}{4}}du
	+ B_1\exp\bbclr{-\frac{t^2}{4}}\int_0^{\abs{t}} u\exp\bbclr{\frac{u^2}{4}}du\\
	& \qquad+ B_2\exp\bbclr{-\frac{t^2}{4}}\int_0^{\abs{t}} u^2\exp\bbclr{\frac{u^2}{4}}du\\
	& = B_0\abs{F(t/2)}
	+ 2B_1(1-e^{-t^2/4}) + 2B_2\abs{t}(1-e^{-t^2/4})\\
	& \leq B_0\abs{F(t/2)}+2(B_1+B_2\abs{t})\min\{1,t^2/4\},
}	
where~$F(t)=e^{-t^2}\int_0^t e^{u^2}du$ is Dawson's function.
Now, Lemma~\ref{lem1} states that
\be{
	\sup_{x\in \IR}\abs{\IP[W\leq x]-\IP[Z\leq x]}
	\leq \frac{1}{\pi}\int_{-T}^T \frac{\abs{\phi(t)-e^{-t^2/2}}}{\abs{t}}dt + \frac{24}{\pi\sqrt{2\pi}T}.
}
Since~$\abs{\phi(t)-e^{-t^2/2}}\leq \abs{R_1(t)}+\abs{R_2(t)}$, we obtain
\bes{
	&\int_{-T}^T \frac{\abs{\phi(t)-e^{t^2/2}}}{\abs{t}}dt\\
	&\qquad\leq \int_{-T}^T \bbbcls{\bbclr{\frac{A_0}{2}\abs{t}+\frac{A_1}{3}t^2}\exp\bbclr{-\frac{t^2}{4}} 
	+ B_0\frac{\abs{F(t/2)}}{\abs{t}} \\
	&\kern18em+2B_1\frac{\min\{1,t^2/4\}}{\abs{t}} + 2B_2\min\{1,t^2/4\}}dt\\
	&\qquad\leq \int_{-\infty}^\infty \bbbcls{\bbclr{\frac{A_0}{2}\abs{t}+\frac{A_1}{3}t^2}\exp\bbclr{-\frac{t^2}{4}} 
	+ B_0\frac{\abs{F(t/2)}}{\abs{t}}}dt\\
	&\kern15.7em+\int_{-T}^T \bbbcls{2B_1\frac{\min\{1,t^2/4\}}{\abs{t}} + 2B_2\min\{1,t^2/4\}}dt\\
	&\qquad\leq 2A_0+\frac{4\sqrt{\pi}}{3}A_1 
	+ \frac{\pi^{3/2}}{2}B_0+2B_1(1+2\log_+(T/2)) + 4B_2T,
}
as long as~$T\leq T_1\wedge T_2 = T_2$. 
Hence, for any~$T\leq T_2 = (1-2A_0)/(2A_1)$,
\bes{
	&\sup_{x\in \IR}\abs{\IP[W\leq x]-\IP[Z\leq x]}\\
	&\qquad\leq {\frac{2}{\pi}A_0+\frac{4}{3\sqrt{\pi}}A_1 
	+ \frac{\sqrt{\pi}}{2}B_0 
	+\frac{2}{\pi}B_1(1+2\log_+(T/2))
	+ \frac{4}{\pi}B_2T + \frac{24}{\pi T\sqrt{2\pi}}},
}
which, after replacing~$T$ by~$1/t$, proves the claim.
\end{proof}

\begin{proof}[Proof of Theorem~\ref{THM2}] Applying~\eq{5} to~$f(x) = e^{itx}$ we have
\besn{\label{10}
	\phi'(t) & = i\IE\clc{Wf(W)} = i\IE\bclc{G(e^{itD}-1)e^{itW}}\\
	& = i\IE\bclc{G(e^{itD}-1)}\phi(t)+i\IE\bclc{\bclr{G(e^{itD}-1)-\IE\bclc{G(e^{itD}-1)}}e^{itW}} \\
	& = -t\bclr{1+a(t)}\phi(t)+b(t)
}
with
\be{
	a(t) = \frac{\IE\bclc{G(e^{itD}-1)}}{it}-1,
	\qquad
	b(t) = i\IE\bclc{\bclr{G(e^{itD}-1)-\IE\bclc{G(e^{itD}-1)}}e^{itW}}.
}
Since~$\IE(GD)=1$, we have
\be{
	a(t) = \frac{\IE\bclc{G(e^{itD}-1-itD)}}{it},
}
and thus,
\be{
	\abs{a(t)}\leq \frac{t}{2}\IE\abs{GD^2}=0.5r_1\abs{t} =: A_1\abs{t}.
}
Moreover,
\bes{
	\abs{b(t)} & \leq \IE\babs{\bclr{\IE^W\bclc{G(e^{itD}-1)}-\IE\bclc{G(e^{itD}-1)}}}	\\ 
	&\leq \sqrt{\Var\IE^W\bclc{G(e^{itD}-1)}}
	\leq r_2\abs{t}=: B_1\abs{t}.
}	
The first claim now follows from Lemma~\ref{lem2} after some straightforward simplifications.

In order to obtain the second claim, add and subtract~$t\IE\bclc{(GD-1)e^{itW}}$ in~\eq{10}, so that
\bes{
	\phi'(t) 
	& = i\IE\bclc{G(e^{itD}-1)}\phi(t)
	- t\IE\bclc{\clr{GD-1}e^{itW}} \\
	& \qquad +i\IE\bclc{\bclr{G(e^{itD}-1-itD)-\IE\bclc{G(e^{itD}-1-itD)}}e^{itW}}.
}
Using the conditions posed on~$\~D$, $S$ and~$W''$,
\bes{
	\IE\bclc{\clr{GD-1}e^{itW}}
	& = \IE\bclc{\bclr{G\~D-S}e^{itW}} \\
	& = \IE\bclc{\bclr{G\~D-S}\bclr{e^{itW}-e^{itW''}}} \\
	& = \IE\bclc{\bclr{G\~D-S}\bclr{1-e^{itD'}}e^{itW}},
}
so that now~$\phi'(t) = -t(1+a(t))\phi(t) + b(t)$ holds with
\ba{
	a(t) & = \frac{\IE\bclc{G(e^{itD}-1-itD)}}{it} - \IE\bclc{G\~D\bclr{e^{itD'}-1}} + \IE\bclc{S\bclr{e^{itD'}-1}} \\
	b(t) & = i\IE\bclc{\bclr{G\bclr{e^{itD}-1-itD}-\IE\bclc{G(e^{itD}-1-itD)}}e^{itW}} \\
	&\qquad + t\IE\bclc{\bclr{G\~D\bclr{e^{itD'}-1}-\IE\bclc{G\~D\bclr{e^{itD'}-1}}}e^{itW}} \\
	&\qquad - t\IE\bclc{\bclr{S\bclr{e^{itD'}-1}-\IE\bclc{S\bclr{e^{itD'}-1}}}e^{itW}},
}
and we have
\ba{
	\abs{a(t)} & \leq t\bclr{\ahalf\IE\abs{GD^2} + \IE\abs{G\~DD'}+\IE\abs{SD'}} = r_3\abs{t} =: A_1\abs{t},\\
	\abs{b(t)} & \leq t^2\bbclr{t^{-2}\sqrt{\Var\IE^W\bclc{G(e^{itD}-1-itD)}} 
	+ t^{-1} \sqrt{\Var\IE^W\bclc{G\~D\bclr{e^{itD'}-1}}}\\
	&\qquad\quad+ t^{-1} \sqrt{\Var\IE^W\bclc{S\bclr{e^{itD'}-1}}} \,} 
	= r_4 t^2 =: B_2t^2.
}
The second claim now follows again from Lemma~\ref{lem2} after some straightforward simplifications.
\end{proof}

\section{Proof of Theorem~\ref{THM1}}

We first need some technical results. Recall that a collection of random variables~$X =(X_1,\dots,X_n)$ is said to be \emph{associated}, if for any two coordinate-wise non-decreasing functions~$f,g:\IR^m\to\IR$, we have
\be{
	\IE f(X)\IE g(X)\leq \IE\clc{f(X)g(X)}
}
whenever the expectations exist.

\begin{lemma}[\cite{Esary1967}]\label{lem3}(1) A collection of independent random variables is associated. (2) Non-decreasing functions of associated random variables are associated. 
\end{lemma}

Let~$[n]:=\{1,\dots,n\}$ denote the set of vertex labels in~$G(n,p)$. For~$1\leq i < j\leq n$, let~$I_{ij}$ be the indicator that there is an edge between vertices~$i$ and~$j$. Let
\be{
	\cE = \bclc{e=\{e_1,e_2\}\subset[n]\,:\,\abs{e}=2},\qquad
	\cT = \bclc{v=\{v_1,v_2,v_3\}\subset[n]\,:\,\abs{v}=3};
}
the first set represent the set of pairs of vertices and the second the set of triples of vertices. We will assume throughout that, in the representations~$e=\{e_1,e_2\}$ and~$v=\{v_1,v_2,v_3\}$, the elements appear in ascending order; we will also use the notations~$I_{e}$ and~$I_{e_1e_2}$ interchageably.
For each~$v\in \cT$, let~$X_{v}=I_{v_1v_2}I_{v_1v_3}I_{v_2v_3}-p^3$ be the centred indicator that there is a triangle between the vertices in~$v$. For concrete indices, such as~$v=\{1,2,3\}$, we will simply write~$X_{123}$ instead of~$X_{\{1,2,3\}}$.

As a direct consequence from Lemma~\ref{lem3} we have the following. 

\begin{corollary}\label{cor1}
(1) For any~$E\subset\cE$, the random variables~$(X_{v})_{v\in\cT}$, given~$(I_{e})_{e\in E}$, are associated. (2) For any~$T_1,T_2\subset \cT$ and any~$E\subset\cE$ the random variables~$\sum_{v\in T_1}X_v$ and~$\sum_{v\in T_2}X_v$, given~$(I_{e})_{e\in E}$, are associated. 
\end{corollary}

\begin{lemma}[{\cite[Lemma~3]{Newman1980}}]\label{lem4} If~$U$ and~$V$ are associated random variables, then, for any complex valued functions~$f$ and~$g$,
\be{
\Cov(f(U),g(V))\leq \norm{f'}\norm{g'}\Cov(U,V),
}
where for complex valued~$f$, $\norm{f} = \sup_{x\in \IR}\abs{f(x)}$.
\end{lemma}

\begin{lemma}[Conditional covariance formula]\label{lem5} For any complex-valued random variables~$U$ and~$V$ we have
\be{
	\Cov(U,V)=\IE\Cov^\cF(U,V) + \Cov(\IE^\cF U,\IE^\cF V).
}
\end{lemma}

\begin{lemma}\label{lem6} For any complex-valued random variables~$U$, $V$, $\~V$, $U'$, $V'$ and~$\~V'$, we have
\be{
	\babs{\Cov\bclr{UV,U'V'}} 
	 \leq \babs{\Cov\bclr{U\~V,U'\~V'}}
	+ R
}
where 
\bes{
	R&= \IE\babs{U\clr{V-\~V}U'V'}+\IE\babs{U\clr{V-\~V}}\IE\babs{U'V'} \\ 
	&\quad + \IE\babs{U\~V'U'\clr{V'-\~V'}} + \IE\babs{U\~V}\IE\babs{U'\clr{V'-\~V'}}
}
\end{lemma}
\begin{proof} Using the linearity properties of the covariance function, we have 
\be{
	 \Cov\bclr{UV,U'V'} 
	 = \Cov\bclr{U\~V,U'\~V'}
	 + \Cov\bclr{U\clr{V-\~V},U'V'}
   +\Cov\bclr{U\~V,U'\clr{V'-\~V'}},
}
from which the claim easily follows.
\end{proof}

\begin{lemma}\label{lem7} There exist universal constants~$c$ and~$C$ such that for all~$n\geq 3$ and for all~$0<p<1$,
\be{
	cs^2(n,p) \leq \Var T\leq Cs^2(n,p),
}
where
\be{
s^2(n,p)= \begin{cases}
			n^4(1-p)	
				& \text{if~$1/2<p<1$,}\\[2ex]
			n^4p^5
				& \text{if~$n^{-1/2}< p \leq 1/2$,}\\[2ex]
			n^3p^3
				& \text{if~$0<p\leq n^{-1/2}$.}
	\end{cases}
}
\end{lemma}
\begin{proof} This follows easily from
\be{
	\Var T 
	 = {n\choose 3}\bclr{p^3(1-p)^3 +3(n-3)p^5(1-p)} 
	 = {n\choose 3}p^3(1-p) \bclr{1+p+p^2 + 3(n-3)p^2}. \qedhere
}
\end{proof}

For the following lemmas, we need some notation. Let~$v_1,\dots,v_k\in\cT$; define
\ben{\label{11}
	M(v_1,\dots,v_k) = \bigcup_{i=1}^k \bclr{\{v_{i,1},v_{i,2}\}\cup\{v_{i,1},v_{i,3}\}\cup\{v_{i,2},v_{i,3}\}}\subset\cE;
}
this represents the set of unique independent edge indicators induced by a collection of vertices. Let
\ben{\label{12}
	\nu_v := \{u\in\cT\,:\, \abs{u\cap v}\geq 2\},
	\qquad
	Y_{v} := \sum_{u\in\nu_v} X_u;
}
this represents the triangle indicators that share at least one edge with~$X_v$ and are thus not independent of~$X_v$. Moreover, for each~$w\in\nu_v$, let
\ben{\label{13}
	\nu_{v,w} = \nu_u\cup\nu_u, 
	\qquad
	Y_{v,w} := \sum_{u\in\nu_v} X_{v,w};
}
this represents the triangle indicators that share at least one edge with either~$X_v$ or~$X_w$ (or both), and are thus not independent of~$(X_v,X_w)$.

\begin{lemma}\label{lem8}Let~$k\geq 1$, and let~$v_1,\dots,v_k\in\cT$. Then
\be{
	\IE\abs{X_{v_1}\cdots X_{v_k}}
	\leq C(k) \min\bclc{1-p,p^m}
}
where~$m=\abs{M(v_1,\dots,v_k)}$.
In addition, let~$w_1,\dots,w_l\in\cT$, $l\geq 1$. Then
\be{
	\Cov\clr{X_{v_1}\cdots X_{v_k},X_{w_1}\cdots X_{w_l}}\leq C(k,l)\min\bclc{1-p,p^m}
}
where~$m=\abs{M(v_1,\dots,v_k,w_1\dots,w_l)}$.
\end{lemma}
\begin{proof} Since all~$X_{v_i}$ are bounded by one, we have, on the one hand, 
\be{
	\IE\abs{X_{v_1}\cdots X_{v_k}}\leq \IE\abs{X_{v_1}} = 2 (1 - p) p^3 (1 + p + p^2) \leq 6(1-p).
}
On the other hand, 
\be{
	\IE\abs{X_{v_1}\cdots X_{v_k}} \leq \sum_{A\subset [k]}
	\IE\bbbclc{\prod_{i\in A}I_{v_{i,1}v_{i,2}}I_{v_{i,1}v_{i,3}}I_{v_{i,2}v_{i,3}}}\prod_{i\in A^c}p^3\leq  2^k p^m.
}
The second part of the lemma is a straightforward consequence of the first part. 
\end{proof}

\begin{lemma}\label{lem9} Let~$v,v'\in\cT$, and let~$w\in\nu_v$ and~$w'\in\nu_{v'}$. Let~$f$ and~$g$ be differentiable, complex-valued functions with~$f(0)=g(0)=0$. Then
\bm{
	\Cov\bclr{X_{v}X_{w}f(Y_{v}),X_{v'}X_{w'}g(Y_{v'})}
	\vee \Cov\bclr{X_{v}X_{w}f(Y_{v,w}),X_{v'}X_{w'}g(Y_{v',w'})} \\
	\leq C\norm{f'}\norm{g'}\min\bclc{n^2(1-p),p^m+np^{m+2}+n^2p^{m+4}},
}
where~$m=\abs{M(v,w,v',w')}$.
\end{lemma}
\begin{proof} We only prove the bound for~$\Cov\bclr{X_{v}X_{w}f(Y_{v,w}),X_{v'}X_{w'}g(Y_{v',w'})}$, as the proof for~$\Cov\bclr{X_{v}X_{w}f(Y_{v}),X_{v'}X_{w'}g(Y_{v'})}$ is essentially the same. Since~$\abs{f(x)}\leq \norm{f'}\abs{x}$ and~$\abs{g(x)}\leq \norm{g'}\abs{x}$,
\bes{
	&\Cov\bclr{X_{v}X_{w}f(Y_{v,w}),X_{v'}X_{w'}g(Y_{v',w'})}\\
	&\qquad \leq \norm{f'}\norm{g'}\bclr{\IE\abs{X_{v}X_{w}Y_{v,w}X_{v'}X_{w'}Y_{v',w'}}+
	\IE\abs{X_{v}X_{w}Y_{v,w}}\IE\abs{X_{v'}X_{w'}Y_{v',w'}}}.
}
We only show how to bound the first expectation above, since the second expression, the product of expectations, can be bounded analogously. 
From Lemma~\ref{lem8}, the bound~$Cn^2(1-p)$ is straightforward since both~$Y_{v,w}$ and~$Y_{v',w'}$ contain order~$n$ summands. 
Now, let~$u\in\nu_{v,w}$ and~$u'\in\nu_{v',w'}$; note that~$\abs{(u\cup u')\setminus (v\cup w \cup v'\cup w')}\leq 2$, since~$u$ is sharing at least two indices with~$v$ or~$w$, and~$u'$ is sharing at least two vertices with~$v'$ or~$w'$. We can thus distinguish three cases.
\begin{enumerate}
\item ``$\abs{(u\cup u')\setminus (v\cup w \cup v'\cup w')}=0$''. This can happen at most~$2{8 \choose 3}$ times. In this case, $M(v,w,v',w',u,u') \geq M(v,w,v',w')=m$, and so, by Lemma~\ref{lem8},
\be{
	\IE\abs{X_{v}X_{w}X_{u}X_{v'}X_{w'}Y_{u'}} \leq Cp^m.
}
\item ``$\abs{(u\cup u')\setminus (v\cup w \cup v'\cup w')}=1$''; this can happen at most order~$n$ times. In this case, $M(v,w,v',w',u,u')\geq M(v,w,v',w')+2$, and so, by Lemma~\ref{lem8},
\be{
	\IE\abs{X_{v}X_{w}X_{u}X_{v'}X_{w'}Y_{u'}} \leq Cp^{m+2}.
}
\item ``$\abs{(u\cup u')\setminus (v\cup w \cup v'\cup w')}=2$''; this can happen at most order~$n^2$ times. In this case, $M(v,w,v',w',u,u')\geq M(v,w,v',w')+4$, and so, by Lemma~\ref{lem8},
\be{
	\IE\babs{X_{v}X_{w}X_{u}X_{v'}X_{w'}Y_{u'}} \leq Cp^{m+4}.
}
\end{enumerate}
Putting the estimates together yields the claim.
\end{proof}

\begin{lemma}\label{lem10}Let~$v,v'\in\cT$, such that~$\abs{v\cap v'}=1$, and let~$w\in\nu_v$ and~$w'\in\nu_{v'}$ be such that~$\abs{(w\cap w')\setminus\clr{v\cap v'})}=0$. Let~$f$ and~$g$ be differentiable, complex-valued functions with~$f(0)=g(0)=0$. Then
\bm{
	\Cov\bclr{X_{v}X_{w}f(Y_{v}),X_{v'}X_{w'}g(Y_{v'})}
	\vee\Cov\bclr{X_{v}X_{w}f(Y_{v,w}),X_{v'}X_{w'}g(Y_{v',w'})}\\
	\leq C\norm{f'}\norm{g'}\min\bclc{n(1-p),p^{m+1}+np^{m+3}},
}
where~$m=\abs{M(v,w,v',w')}$.
\end{lemma}
\begin{proof} As in the proof of Lemma~\ref{lem9}, we only consider~$\Cov\bclr{X_{v}X_{w}f(Y_{v,w}),X_{v'}X_{w'}g(Y_{v',w'})}$. Throughout the proof, we suppress the dependence on~$v$, $w$, $v'$ and~$w'$ in many places, since they are fixed. Define the sets
\ba{
	\eta & = \bclc{u\in\nu_{v,w}\setminus\{v,w\}\,:\,\text{$\abs{u\cap(v'\cup w')}=1$}},\\
	\eta' & = \bclc{u\in\nu_{v',w'}\setminus\{v',w'\}\,:\,\text{$\abs{u\cap(v\cup w)}=1$}}
}
(note that, under the conditions imposed on~$v$, $w$, $v'$ and~$w'$, we cannot have~$\abs{u\cap(v'\cup w')}>1$). The set~$\eta$ represents the set of indices of those triangle indicators in~$\nu_{v,w}$ which are not equal to~$v$ and~$w$ and which have one vertex in the set~$v'\cup w'$, and likewise, the set~$\eta'$ represents the set of indices of those triangle indicators in~$\nu_{v',w'}$ which are not equal to~$v'$ and~$w'$ and which have one vertex in the set~$v\cup w$. It is important to note that for each~$u\in \eta\cup\eta'$, we have 
\ben{\label{14}
	\abs{M(v,w,v',w',u)}\geq \abs{M(v,w,v',w')}+1 = m+1.
}
Now, let
\be{
	Z_{v,w} = \sum_{u\in \eta} X_u,\qquad
	Z_{v',w'} = \sum_{u\in \eta'} X_u,
}
and let
\be{
	\~Y_{v,w}  = Y_v - Z_v, 
	\qquad
	\~Y_{v',w'}  = Y_{v'} - Z_{v'}.
}
The sum~$\~Y_{v,w}$ consists of those centred triangle indicators which are~$(i)$ equal to~$X_v$ or~$X_w$, or~$(ii)$ composed entirely of vertices from~$v\cup w$, or~$(iii)$ share one edge with~$X_v$ or~$X_w$, but whose third vertex is not in~$v\cup w\cup v'\cup w'$ (and the analogous statement holds for~$\~Y_{v',w'}$ with cases~$(i')$, $(ii')$ and~$(iii')$). Note that if~$X_u$ is a summand in~$\~Y_{v,w}$ from~$(iii)$ above, and if~$X_{u'}$ is a summand in~$\~Y_{v',w'}$ from~$(iii')$, and if~$u$ and~$u'$ are such that the respective third vertex is equal, then
\ben{\label{15}
	\abs{M(v,w,v',w',u,u')}\geq \abs{M(v,w,v',w')}+3 = m+3.
}
We now apply Lemma~\ref{lem6} with
\ba{
	U&=X_{v}X_{w}, &V&=f(Y_{v,w}), &\~V&=f(\~Y_{v,w}),\\
	U'&=X_{v'}X_{w'}, &V'&=f(Y_{v',w'}), &\~V'&=f(\~Y_{v',w'}).
}
Note that, since~$\abs{f(x)}\leq \norm{f'}\abs{x}$ and~$\abs{g(x)}\leq \norm{g'}\abs{x}$, we have
\be{
	\abs{V-\~V}\leq \norm{f'}\abs{Z_{v,w}},
	\qquad
	\abs{V'-\~V'}\leq \norm{g'}\abs{Z_{v',w'}}.
}
We obtain
\be{
	\babs{\Cov\bclr{X_{v}X_{w}f(Y_{v,w}),X_{v'}X_{w'}g(Y_{v',w'})}}
	 \leq \babs{\Cov\bclr{X_{v}X_{w}f(\~Y_{v,w}),X_{v'}X_{w'}g(\~Y_{v',w'})}} + R,
}
where
\bes{
	R
	&\leq \IE\abs{X_{v}X_{w}Z_{v,w}X_{v'}X_{w'}Y_{v',w'}}+\IE\abs{X_{v}X_{w}Z_{v,w}}\IE\abs{X_{v'}X_{w'}Y_{v',w'}} \\ 
	&\quad + \IE\abs{X_{v}X_{w}\~Y_{v,w}X_{v'}X_{w'}Z_{v',w'}} + \IE\abs{X_{v}X_{w}\~Y_{v,w}}\IE\abs{X_{v'}X_{w'}Z_{v',w'}}.
}
From this,~\eq{14},~\eq{15} and Lemma~\ref{lem8}, it is not difficult to see that
\be{
	R\leq C\min\bclc{n(1-p),p^{m+1}+np^{m+3}}.
}
Now, let 
\ben{\label{15a}
	M^*(v,w) = \bclc{u\in \cT\,:\,u\subset(v\cup w)}
}
and define the~$\sigma$-algebras
\ben{\label{16}
	\cF_1  = \sigma(I_e; e\in M^*(v,w)),
	\qquad
	\cF_2  = \sigma(I_e; e\in M^*(v',w'));
}
note that~$X_{v}X_{w}$ is~$\cF_1$-measurable and~$X_{v'}X_{w'}$ is~$\cF_2$-measurable, and that~$\cF_1$ and~$\cF_2$ are independent of each other. With~$\cF=\sigma(\cF_1,\cF_2)$, we apply Lemma~\ref{lem5} and obtain
\bes{
	\Cov\bclr{X_{v}X_{w}f(\~Y_{v,w}),X_{v'}X_{w'}g(\~Y_{v',w'})}
	& = \Cov\bclr{X_{v}X_{w}\IE^\cF f(\~Y_{v,w}),X_{v'}X_{w'}\IE^\cF g(\~Y_{v',w'})}\\
	&\quad+\IE\Cov^\cF\bclr{X_{v}X_{w}f(\~Y_{v,w}),X_{v'}X_{w'}g(\~Y_{v',w'})}\\
	& =: a_1 + a_2.
}
Now, by construction of~$\~Y_{v,w}$ and~$\~Y_{v',w'}$, we have~$\IE^\cF f(\~Y_{v,w}) = \IE^{\cF_1} f(\~Y_{v,w})$ and~$\IE^\cF g(\~Y_{v',w'}) = \IE^{\cF_2} g(\~Y_{v',w'})$, from which~$a_1=0$ is immediate. In order to bound~$a_2$, we write
\be{
	\Cov^\cF\bclr{X_{v}X_{w}f(\~Y_{v,w}),X_{v'}X_{w'}g(\~Y_{v',w'})}
	 = X_{v}X_{w}X_{v'}X_{w'}\Cov^\cF\bclr{f(\~Y_{v,w}),g(\~Y_{v',w'})}.
}
Using Corollary~\ref{cor1} and Lemma~\ref{lem4}, we have
\be{
	\babs{\Cov^\cF\bclr{f(\~Y_{v,w}),g(\~Y_{v',w'})}}
	\leq \norm{f'}\norm{g'}\Cov^\cF\bclr{\~Y_{v,w},\~Y_{v',w'}}.
}
We now can write
\bes{
	\Cov^\cF\bclr{\~Y_{v,w},\~Y_{v',w'}} 
	& = \sum_{\substack{e\in M(v,w), \\ e'\in M(v',w')}}\sum_{i\in [n]\setminus(v\cup w\cup v' \cup w')}
	\Cov^\cF\bclr{X_{e_1e_2i},X_{e_1'e_2'i}} \\
	& \quad + \sum_{\substack{e\in M(v,w), \\ e'\in M(v',w')}}\sum_{\substack{i,j\in [n]\setminus(v\cup w\cup v' \cup w'),\\ i\neq j}}
	\Cov^\cF\bclr{X_{e_1e_2i},X_{e_1'e_2'j}}.
}
From the assertions we conclude that~$e\in M(v,w)$ and~$e'\in M(v',w')$ implies~$e\neq e'$. Hence, we can use independence to find that the covariances in the second sum all vanish, and moreover, that~$M(v,w,v',w',\clc{e_1,e_2,i},\clc{e_1',e_2',i})\geq M(v,w,v',w')+3$. Thus, Lemma~\ref{lem8} yields
\bm{
	a_2\leq\IE\bbbabs{X_vX_wX_{v'}X_{w'}\sum_{\substack{e\in M(v,w), \\ e'\in M(v',w')}}\sum_{i\in [n]\setminus(v\cup w\cup v' \cup w')}
	\Cov^\cF\bclr{X_{e_1e_2i},X_{e_1'e_2'i}}}\\
	\leq C \min\{n(1-p),np^{m+3}\}.
}
Collecting all the estimates gives the final bound. 
\end{proof}

\begin{lemma}\label{lem11} Let~$v,v'\in\cT$ be such that~$\abs{v\cap v'}=0$, and let~$w\in\nu_v$ and~$w'\in\nu_{v'}$ be such that~$\babs{(v\cup w)\cap (v'\cup w')}=0$. Let~$f$ and~$g$ be differentiable, complex-valued functions with~$f(0)=g(0)=0$. Then
\bm{
	\Cov\bclr{X_{v}X_{w}f(Y_{v}),X_{v'}X_{w'}g(Y_{v'})}
	\vee\Cov\bclr{X_{v}X_{w}f(Y_{v,w}),X_{v'}X_{w'}g(Y_{v',w'})}\\
	\leq C\norm{f'}\norm{g'}\min\bclc{1-p,p^{m+3}},
}
where~$m=\abs{M(v,w,v',w')}$.
\end{lemma}
\begin{proof} As in Lemma~\ref{lem9}, we only consider~$\Cov\bclr{X_{v}X_{w}f(Y_{v,w}),X_{v'}X_{w'}g(Y_{v',w'})}$. Let~$\cF_1$, $\cF_2$ and~$\cF$ be as in~\eq{16}.
Applying Lemma~\ref{lem5},
\besn{\label{17}
	\Cov\bclr{X_{v}X_{w}f(Y_{v,w}),X_{v'}X_{w'}g(Y_{v',w'})} 
&=\Cov\bclr{X_{v}X_{w}\IE^{\cF}f(Y_{v,w}),X_{v'}X_{w'}\IE^{\cF}g(Y_{w,w'})} \\ 
	&\quad+\IE\bclc{X_{v}X_{w}X_{v'}X_{w'}\Cov^\cF\bclr{f(Y_{v,w}),g(Y_{v',w'})}} \\
	& =: a_1 + a_2.
}
Now, since~$\IE^{\cF}f(Y_{v,w}) = \IE^{\cF_1}f(Y_{v,w})$
and~$\IE^{\cF}f(Y_{v',w'}) = \IE^{\cF_2}f(Y_{v',w'})$, we have~$a_1=0$. By Corollary~\ref{cor1} and Lemma~\ref{lem4},
\be{
	\babs{\Cov^\cF\bclr{f(Y_{v,w}),g(Y_{v',w'})}}
	\leq C\norm{f'}\norm{g'}\Cov^\cF\bclr{Y_{v,w},Y_{v',w'}}.
}
We can write
\bes{
	\Cov^\cF\bclr{Y_{v,w},Y_{v',w'}} 
	& = \sum_{\substack{e\in M(v,w), \\ e'\in M(v',w')}}\sum_{\substack{i\in( v' \cup w'),\\ j\in(v\cup w)}}
	\Cov^\cF\bclr{X_{e_1e_2i},X_{e_1'e_2'j}}\\
	&\quad \quad +\sum_{\substack{e\in M(v,w), \\ e'\in M(v',w')}}\sum_{\substack{i,j\in [n]:\\ i\not\in( v' \cup w')\text{ or}\\ j\not\in(v\cup w)}} \Cov^\cF\bclr{X_{e_1e_2i},X_{e_1'e_2'j}}.
}
From the assertions we conclude that~$e\in M(v,w)$ and~$e'\in M(v',w')$ implies~$\abs{e\cap e'}=0$. Hence, we can use independence to find that the covariances in the second sum all vanish. For the first sum, note that~$M(v,w,v',w',\clc{e_1,e_2,i},\clc{e_1',e_2',j})\geq M(v,w,v',w')+3$. Thus, Lemma~\ref{lem8} yields
\bm{
	a_2\leq 
	\IE\bbbabs{X_{v}X_{w}X_{v'}X_{w'}\sum_{\substack{e\in M(v,w), \\ e'\in M(v',w')}}\sum_{\substack{i\in( v' \cup w'),\\ j\in(v\cup w)}}
	\Cov^\cF\bclr{X_{e_1e_2i},X_{e_1'e_2'j}}}\\
	\leq C \min\{1-p,p^{m+3}\}.
}
Collecting all the estimates gives the final bound. 
\end{proof}

\begin{lemma}\label{lem12} Let~$v,v'\in\cT$ be such that~$\abs{v\cap v'}=0$, and let~$w\in \nu_v$ and~$w'\in \nu_{v'}$ be such that~$\babs{(v\cup w)\cap (v'\cup w')}=1$. Let~$f$ and~$g$ be differentiable, complex-valued functions with~$f(0)=g(0)=0$. Then
\bm{
	\Cov\bclr{X_{v}X_{w}f(Y_{v}),X_{v'}X_{w'}g(Y_{v'})}
	\vee \Cov\bclr{X_{v}X_{w}f(Y_{v,w}),X_{v'}X_{w'}g(Y_{v',w'})}\\
	\leq C\norm{f'}\norm{g'}\min\bclc{n(1-p),p^{m+1}+np^{m+3}},
}
where~$m=\abs{M(v,w,v',w')}$.
\end{lemma}
\begin{proof} The proof is similar to that of Lemma~\ref{lem10} and therefore ommited. 
\end{proof}

\begin{proof}[Proof of Theorem~\ref{THM1}] 
\noindent\textbf{Construction of Stein coupling.\enskip}
Our proof is based on a Stein coupling that is equivalent to what was implicitly used by \cite{Barbour1989}. Let 
\be{
	W = \frac{1}{\sigma}\sum_{v\in\cT}X_{v}
}
be the scaled sum of centred triangle indicator random variables. Let~$V$ be an independent random variable,  uniformly distributed on~$\cT$; set  
\be{
	G=-{n\choose 3}\frac{X_V}{\sigma},\qquad W'=W - \frac{Y_{V}}{\sigma}. 
}
Using the fact that~$W - Y_{v}/\sigma$ and~$X_{v}$ are independent, it is straightforward to verify that~$(W,W',G)$ is a Stein coupling, and we have~$D=-Y_{V}/\sigma$. 

\medskip

\noindent\textbf{Construction of extended Stein coupling.\enskip} In order to construct~$\~D$, $S$ and~$W''$, recall the definition of~$\nu_v$ from~\eq{12};
note that~$\abs{\nu_v}=3(n-3)+1$ and~$Y_v = \sum_{u\in\nu_v}X_u$. Given~$V$, let~$V'$ be a random variable, independent of all else and distributed uniformly on~$\nu_v$. Set
\be{
	\~D := -\frac{3(n-3)+1}{\sigma}X_{V'}.
}
It is straightforward to check that 
\be{
	\IE\bclc{\~D\given (I_{e})_{e\in\cE},V}
	= -\frac{Y_{V}}{\sigma} = D,
}	
which implies that~$\IE^W(G\~D) = \IE^W(GD)$ since~$W$ is measurable with respect to~$(I_{e})_{e\in\cE}$.
Let
\be{
	S := {n\choose 3}\frac{3(n-3)+1}{\sigma^2}\times\begin{cases}
	\Var X_{123} & \text{if~$V'=V$,} \\
	\Cov\bclr{X_{123},X_{124}}& \text{if~$V'\neq V$,}
	\end{cases}
}
and it is easy to verify that~$\IE^W S = 1$. For~$v\in\cT$ and ~$w\in\nu_v$, recall the definition of~$\nu_{v,w}$ in~\eq{13}; let
\be{
	Y_{v,w} = \sum_{u\in\nu_{v,w}}X_u,
}
and moreover, let
\be{
	W'' = W - \frac{Y_{V,V'}}{\sigma}.
}
Using independence between~$W-Y_{v,w}/\sigma$ and~$X_{v}X_w$ it follows that
\bes{
	&\IE\bclc{G\~D\given V,W''} \\
	&\qquad = {n\choose 3}\frac{1+3(n-3)}{\sigma^2}\begin{cases}
	\IE\bclc{X_{V}^2} & \text{if~$V=V'$,} \\[2ex]
	\IE\bclc{X_{V}X_{V'}} & \text{if~$V\neq V'$} \\
	\end{cases} \\
	&\qquad = S,
}
which implies that~$\IE^{W''}(G\~D) = \IE^{W''}S$, and we have 
\be{
	D' = -\frac{1}{\sigma}Y_{V,V'}.
}
Thus, all the conditions in Theorem~\ref{THM2} are satisfies, and it remains to bound~$r_3$ and~$r_4$.

\smallskip

\noindent\textbf{Bounding~$\boldsymbol{r_3}$.\enskip} 
In what follows, $C$ denotes a constant that does not depend on~$n$ and~$p$, and which can change from line to line. First, write
\be{
	r_3  = \ahalf\IE\abs{GD^2} + \IE\abs{G\~DD'}  + \IE\abs{SD'}
	=: \ahalf r_{3,1} + r_{3,2} + r_{3,3}.
}
Using some obvious symmetries, and bounding terms like~$n-1$, $n-2$ and so forth by~$n$, we have
\bes{
	r_{3,1} &
	\leq \frac{Cn^3}{\sigma^3}\IE\babs{X_{V}Y_{V}^2}
	= \frac{Cn^3}{\sigma^3}\IE\babs{X_{123}Y_{123}^2}.
}
On the one hand, we have the simple bound~$\IE\babs{X_{123}Y_{123}^2}\leq n^2\IE\abs{X_{123}}\leq Cn^2(1-p)$, and on the other hand, by means of Lemma~\ref{lem8}, it is straightforward to see that 
\bes{
	\IE\babs{X_{123}Y_{123}^2}\leq C\bclr{p^3 + np^5+n^2p^7};
}
hence,
\be{
	r_{3,1} \leq \frac{C}{\sigma^3}\min\bclc{n^5(1-p),n^3p^3+n^4p^5+n^5p^7}
} 
Moreover,
\bes{
	r_{3,2}  
	& = \IE\abs{G\~DD'} \\
	& \leq \frac{Cn^4}{\sigma^3} 
		\IE\abs{X_V X_{V'} Y_{V,V'}}\\
	& \leq \frac{Cn^4}{\sigma^3} 
		\IE\babs{X_V^2 Y_{V} \I[V=V']+X_V X_{V'} Y_{V,V'}\I[V\neq V']}\\
	& \leq \frac{Cn^3}{\sigma^3} 
		\IE\babs{X_{123}^2Y_{123}} + \frac{Cn^4}{\sigma^3}\IE\abs{X_{123}X_{124}Y_{123,124}}.
}
On the hand hand, we can bound both~$\IE\babs{X_{123}^2Y_{123}}$ and~$\IE\babs{X_{123}X_{124}Y_{123,124}}$ by~$Cn(1-p)$. On the other hand, we have~$\IE\babs{X_{123}^2Y_{123}}\leq C(p^3+np^5)$ and~$\IE\babs{X_{123}X_{124}Y_{123,124}}\leq C(p^5+np^7)$ from Lemma~\ref{lem8}, so that
\be{ 
	 r_{3,2}\leq \frac{C}{\sigma^3}\min\bclc{n^5(1-p),n^3p^3+n^4p^5+n^5p^7}.
}
Finally, 
\bes{
	r_{3,3}  
	& = \IE\abs{SD'} \\
	& \leq \frac{Cn^4}{\sigma^3} 
		\IE\babs{\I[V=V']Y_{V}\Var(X_{123}) + \I[V\neq V']Y_{V,V'}\Cov\clr{X_{123},X_{124}}}\\
	& \leq \frac{Cn^3}{\sigma^3}\Var(X_{123}) 
		\,\IE\abs{Y_{123}} + \frac{Cn^4}{\sigma^3}\IE\abs{Y_{123,124}}\Cov\clr{X_{123},X_{124}}.
}
On the one hand, we can bound both~$\Var X_{123}$ and~$\Cov\bclr{X_{123},X_{124}}$ by~$C(1-p)$, and both~$\IE\babs{Y_{123}}$ and~$\IE\babs{Y_{1234}}$ by~$n$. On the other hand, we have the bounds~$\Var X_{123}\leq Cp^3$, $\Cov\clr{X_{123},X_{124}}\leq Cp^5$, $\IE\abs{Y_{123}}\leq Cnp^3$ and~$\IE\abs{Y_{123,124}}\leq Cnp^3$. This yields
\be{
	r_{3,3} \leq\frac{C}{\sigma^3}\min\bclc{n^5(1-p),n^4p^6+n^5p^8}.
}
Combining the bounds on~$r_{3,1}$, $r_{3,2}$ and~$r_{3,3}$, we obtain 
\ben{\label{18}
	r_3 \leq\begin{cases}
		\ds\frac{Cn^5(1-p)}{\sigma^3}&\text{if~$1/2<p<1$,}\\[2ex]
		\ds\frac{Cn^5p^7}{\sigma^3}& \text{if~$n^{-1/2}<p\leq 1/2$,}\\[2ex]
		\ds\frac{Cn^3p^3}{\sigma^3}& \text{if~$0<p\leq n^{-1/2}$.}
	\end{cases}
}
For later use, we let~$\~r_3$ be equal to the right hand side of~\eq{18}.
\smallskip

\noindent\textbf{Bounding~$\boldsymbol{r_4}$.\enskip} 
First write
\be{
	r_4 =: \sup_{t\neq0}\frac{r_{4,1}(t)^{1/2}}{t^2}
	 + \sup_{t\neq0}\frac{r_{4,2}(t)^{1/2}}{\abs{t}} 
	+ \sup_{t\neq0}\frac{r_{4,3}(t)^{1/2}}{\abs{t}}.
}
From now on, we consider~$t$ as being fixed and, thus, drop it from our notation. Moreover, instead of conditioning~$G(e^{itD}-1-itD)$ on~$W$, we may condition on~$(I_e)_{e\in\cE}$, making the corresponding variances in~$r_4$ only larger. Now, define the function
\ben{\label{19}
	\phi(x) = \begin{cases}	\ds\frac{e^{ix}-1-ix}{x} & \text{if~$x\neq 0$,}\\[2ex]
	0 &\text{if~$x=0$,}
	\end{cases}
}
and for later use, note that
\ben{\label{20}
	\abs{\phi(x)-\phi(y)}\leq \frac{\abs{x-y}}{2}.
}
Hence,
\bes{
	r_{4,1} &\leq \frac{1}{\sigma^2}\Var\sum_{v}X_{v}\bclr{e^{tY_{v}/\sigma}-1-itY_{v}/\sigma}\\
	&	\leq \frac{Ct^2 n^3}{\sigma^4}\bclr{r_{4,1,1}+nr_{4,1,2} + n^2r_{4,1,3} + n^3r_{4,1,4}},
}
where
\ba{
	r_{4,1,1} &= \Cov\bclr{X_{123}Y_{123}\phi\bclr{tY_{123}/\sigma},X_{123}Y_{123}\phi\bclr{tY_{123}/\sigma}},\\
	r_{4,1,2} &= \babs{\Cov\bclr{X_{123}Y_{123}\phi\bclr{tY_{123}/\sigma},X_{124}Y_{124}\phi\bclr{tY_{124}/\sigma}}},\\
	r_{4,1,3} &= \babs{\Cov\bclr{X_{123}Y_{123}\phi\bclr{tY_{123}/\sigma},X_{145}Y_{145}\phi\bclr{{tY_{145}/\sigma}}}},\\
	r_{4,1,4} &= \babs{\Cov\bclr{X_{123}Y_{123}\phi\bclr{tY_{123}/\sigma},X_{456}Y_{456}\phi\bclr{{tY_{456}/\sigma}}}}.
}
Combining all estimates from Tables~\ref{tab1} to~\ref{tab4}, we obtain
\be{
	r_{4,1} \leq\begin{cases}
		\ds\frac{Ct^4n^8(1-p)}{\sigma^6}&\text{if~$1/2<p<1$,}\\[2ex]
		\ds\frac{Ct^4n^8p^{13}}{\sigma^6}& \text{if~$n^{-1/2}<p\leq 1/2$,}\\[2ex]
		\ds\frac{Ct^4n^3p^3}{\sigma^6}& \text{if~$0<p\leq n^{-1/2}$.}
	\end{cases}
}
We continue to bound~$r_{4,2}$; to this end, define
\ben{\label{21}
	\psi(x) = e^{ix}-1,
}
and, again for later use, note that
\ben{\label{22}
	\abs{\psi(x)-\psi(y)}\leq \abs{x-y}.
}
Hence,
\bes{
	r_{4,2} &\leq \frac{1}{\sigma^4}\Var\sum_{v\in\cT}X_{v}\sum_{w\in\eta_v}X_w\bclr{e^{tY_{v,w}/\sigma}-1}\\
	&	\leq \frac{C n^3}{\sigma^4}\bclr{r_{4,2,1}+nr_{4,2,2} + n^2r_{4,2,3} + n^3r_{4,2,4}},
}
where
\ba{
	r_{4,2,1} &= \Cov\bclr{X_{123}\tssum_{w\in\nu_{123}}X_w\psi\bclr{tY_{123,w}/\sigma},X_{123}\tssum_{w\in\nu_{123}}X_w\psi\bclr{tY_{123,w}/\sigma}},\\
	r_{4,2,2} &= \babs{\Cov\bclr{X_{123}\tssum_{w\in\nu_{123}}X_w\psi\bclr{tY_{123,w}/\sigma},X_{124}\tssum_{w\in\nu_{124}}X_w\psi\bclr{tY_{124,w}/\sigma}}},\\
	r_{4,2,3} &= \babs{\Cov\bclr{X_{123}\tssum_{w\in\nu_{123}}X_w\psi\bclr{tY_{123,w}/\sigma},X_{145}\tssum_{w\in\nu_{145}}X_w\psi\bclr{tY_{145,w}/\sigma}}},\\
	r_{4,2,4} &= \babs{\Cov\bclr{X_{123}\tssum_{w\in\nu_{123}}X_w\psi\bclr{tY_{123,w}/\sigma},X_{456}\tssum_{w\in\nu_{456}}X_w\psi\bclr{tY_{456,w}/\sigma}}}.
}
Combining all estimates from Tables~\ref{tab1} to~\ref{tab4}, we obtain
\be{
	r_{4,2} \leq\begin{cases}
		\ds\frac{Ct^2n^8(1-p)}{\sigma^6}&\text{if~$1/2<p<1$,}\\[2ex]
		\ds\frac{Ct^2n^8p^{13}}{\sigma^6}& \text{if~$n^{-1/2}<p\leq 1/2$,}\\[2ex]
		\ds\frac{Ct^2n^3p^3}{\sigma^6}& \text{if~$0<p\leq n^{-1/2}$.}
	\end{cases}
}
We proceed to bound~$r_{4,3}$. Let~$\sigma_{v,w} = \Cov(X_v,X_w)$, we have
\bes{
	r_{4,3} &\leq \frac{1}{\sigma^4}\Var\sum_{v\in\cT}\sum_{w\in\eta_v}\sigma_{v,w}\bclr{e^{tY_{v,w}/\sigma}-1}\\
	&	\leq \frac{C n^3}{\sigma^4}\bclr{r_{4,3,1}+nr_{4,3,2} + n^2r_{4,3,3} + n^3r_{4,3,4}},
}
where
\ba{
	r_{4,3,1} &= \Cov\bclr{\tssum_{w\in\nu_{123}}\sigma_{123,w}\psi\bclr{tY_{123,w}/\sigma},\tssum_{w\in\nu_{123}}\sigma_{123,w}\psi\bclr{tY_{123,w}/\sigma}},\\
	r_{4,3,2} &= \babs{\Cov\bclr{\tssum_{w\in\nu_{123}}\sigma_{123,w}\psi\bclr{tY_{123,w}/\sigma},\tssum_{w\in\nu_{124}}\sigma_{124,w}\psi\bclr{tY_{124,w}/\sigma}}},\\
	r_{4,3,3} &= \babs{\Cov\bclr{\tssum_{w\in\nu_{123}}\sigma_{123,w}\psi\bclr{tY_{123,w}/\sigma},\tssum_{w\in\nu_{145}}\sigma_{145,w}\psi\bclr{tY_{145,w}/\sigma}}},\\
	r_{4,3,4} &= \babs{\Cov\bclr{\tssum_{w\in\nu_{123}}\sigma_{123,w}\psi\bclr{tY_{123,w}/\sigma},\tssum_{w\in\nu_{456}}\sigma_{456,w}\psi\bclr{tY_{456,w}/\sigma}}}.
}

Bounds can also be obtained from Tables~\ref{tab1} to~\ref{tab4}, since~$\sigma_{v,w}\sigma_{v',w'}\leq C\min\{1-p,p^m\}$ where~$m=\abs{M(v,w,v',w')}$, so that bounds on~$r_{4,2,1}$ to~$r_{4,2,4}$ are also bounds on~$r_{4,3,1}$ to~$r_{4,3,4}$. Thus,

\be{
	r_{4,3} \leq\begin{cases}
		\ds\frac{Ct^2n^8(1-p)}{\sigma^6}&\text{if~$1/2<p<1$,}\\[2ex]
		\ds\frac{Ct^2n^8p^{13}}{\sigma^6}& \text{if~$n^{-1/2}<p\leq 1/2$,}\\[2ex]
		\ds\frac{Ct^2n^3p^3}{\sigma^6}& \text{if~$0<p\leq n^{-1/2}$.}
	\end{cases}
}
Putting the estimates together, we obtain
\be{
	r_{4} \leq\begin{cases}
		\ds\frac{Cn^4(1-p)^{1/2}}{\sigma^3}&\text{if~$1/2<p<1$,}\\[2ex]
		\ds\frac{Cn^4p^{13/2}}{\sigma^3}& \text{if~$n^{-1/2}<p\leq 1/2$,}\\[2ex]
		\ds\frac{Cn^{3/2}p^{3/2}}{\sigma^3}& \text{if~$0<p\leq n^{-1/2}$.}
	\end{cases}
}
From Lemma~\ref{lem7}, 
\be{
	\sigma^2\geq \begin{cases}
		Cn^4(1-p)&\text{if~$1/2<p<1$,}\\[2ex]
		Cn^4p^{5}& \text{if~$n^{-1/2}<p\leq 1/2$,}\\[2ex]
		Cn^{3}p^{3}& \text{if~$0<p\leq n^{-1/2}$.}
	\end{cases}
}
Applying Theorem~\ref{THM2}, the final bound follows.
\end{proof}

\subsection{Covariance estimates}

In the following tables, we provide bounds necessary for the proof of Theorem~\ref{THM1}. We illustrate how to read the tables by means of~$r_{4,1,3}$ from the proof of Theorem~\ref{THM1}. We have
\bes{
r_{4,1,3} &= \babs{\Cov\bclr{X_{123}Y_{123}\phi\bclr{tY_{123}/\sigma},X_{145}Y_{145}\phi\bclr{{tY_{145}/\sigma}}}};
}
expanding the covariance over~$Y_{123}$ and~$Y_{145}$, we have
\be{
	r_{4,1,3}\leq 
	\sum_{u\in\nu_{123}}\sum_{u'\in\nu_{145}}
	\babs{\Cov\bclr{X_{123}X_{u}\phi\bclr{tY_{123}/\sigma},X_{145}X_{u'}\phi\bclr{{tY_{145}/\sigma}}}}.
}
Table~\ref{tab3} now gives bound on these covariance terms for all possible combinations of~$u$ and~$u'$ (modulo symmetries). The solid thick line represents the variable~$X_{123}$, the dashed thick line the variable~$X_{u}$, the solid thin line the variable~$X_{145}$ and the dashed thin line the variable~$X_{u'}$. For example, for~$u=\{2,3,4\}$ and~$u'=\{4,5,6\}$ (15th row in Table~\ref{tab3}) we find that
\bm{
	\babs{\Cov\bclr{X_{123}X_{234}\phi\bclr{tY_{123}/\sigma},X_{145}X_{456}\phi\bclr{{tY_{145}/\sigma}}}}\\
	\leq \frac{Ct^2}{\sigma^2}\min\bclc{n^2(1-p),p^9+np^{11}+n^2p^{13}}.
}
This covariance term appears order~$n$ times in~$r_{4,1,3}$, since vertex~$6$ represents a generic vertex different from~$\{1,2,3,4,5\}$, and the bound was obtained through Lemma~\ref{lem9}.

The number of occurrences and all bounds provided are up to combinatorial constants, which are independent of~$n$ and~$p$.


\captionsetup{labelfont={bf,small},font={small}}

\tikzset{
 every picture/.style={xscale=0.4, yscale=0.4},
 vertex/.style={circle, fill, inner sep=1.5pt},
 tr1/.style={line width =   0.9pt},
 tr2/.style={line width =   0.9pt, densely dotted},
 tr3/.style={line width = 0.5pt},
 tr4/.style={line width = 0.5pt, densely dotted},
 dontbend/.style={bend left=0},
 halfbendright/.style={bend right=6},
 halfbendleft/.style={bend left=6},
 bendright/.style={bend right=10},
 bendleft/.style={bend left=10},
 bendbendright/.style={bend right=26},
 bendbendleft/.style={bend left=26},
 bendhalfbendright/.style={bend right=20},
 bendhalfbendleft/.style={bend left=20},
 bendbendbendright/.style={bend right=40},
 bendbendbendleft/.style={bend left=40},
 bendbendbendbendright/.style={bend right=55},
 bendbendbendbendleft/.style={bend left=55}
}

\def\db{dontbend}  
\def\rhb{halfbendright} 
\def\lhb{halfbendleft}  
\def\rb{bendright} 
\def\lb{bendleft}  
\def\rbb{bendbendright}
\def\lbb{bendbendleft} 
\def\rbhb{bendhalfbendright}
\def\lbhb{bendhalfbendleft} 
\def\rbbb{bendbendbendright}
\def\lbbb{bendbendbendleft} 
\def\rbbbb{bendbendbendbendright}
\def\lbbbb{bendbendbendbendleft} 

\def\tr[#1]#2#3#4#5#6#7{
\draw (#2) [#1, #3] to (#4)
			(#4) [#1, #5] to (#6)
			(#6) [#1, #7] to (#2);
}


\begin{center}\setlength{\tabcolsep}{.5em}
\footnotesize
\begin{longtable}{
	>{\centering}m{0.8in}
	>{\centering}m{0.5in}
	>{\centering}m{0.65in}
	>{\centering}m{1.0in}
	>{\centering\arraybackslash}m{0.5in}
}
\caption{\label{tab1} All possible combinations of triangle counts that can occur when expanding~$r_{4,1,1}$ and~$r_{4,2,1}$, along with the bounds on the  corresponding summands. } \\
\toprule
\bf Pattern & 
\bf occur\-rences & 
\bf bound for large~$\boldsymbol{p}$&
\bf bound for small~$\boldsymbol{p}$&
\bf Lemma used\\
\midrule
\endfirsthead

\multicolumn{4}{c}%
{\tablename~\thetable\ (continued from previous page)} \\
\midrule
\bf Pattern & 
\bf occur\-rences & 
\bf bound for large~$\boldsymbol{p}$&
\bf bound for small~$\boldsymbol{p}$&
\bf Lemma used\\
\midrule
\endhead

\multicolumn{5}{r}{(continued on next page)}\\
\endfoot

\bottomrule
\endlastfoot

\begin{tikzpicture}
\path ( 90:1) node[vertex,label=above:1] (1) {} to 
      (210:1) node[vertex,label=left:2 ] (2) {} to 
      (330:1) node[vertex,label=right:3] (3) {};
\tr[tr1] 1 \rb 2 \rb 3 \rb
\tr[tr2] 1 \rbb 2 \rbb 3 \rbb
\tr[tr3] 1 \lb 2 \lb 3 \lb
\tr[tr4] 1 \lbb 2 \lbb 3 \lbb
\end{tikzpicture} &
$1$&
$n^2(1-p)$&
$p^3+np^5+n^2p^7$&
Lem.~\ref{lem9} \\
\begin{tikzpicture}
\path ( 90:1) node[vertex,label=above:1] (1) {} to 
      (210:1) node[vertex,label=left:2 ] (2) {} to 
      (330:1) node[vertex,label=right:3] (3) {} to
      (150:2) node[vertex,label=above:4] (4) {};
\tr[tr1] 1 \rb 2 \rb 3 \rb
\tr[tr2] 1 \rbb 2 \lb 4 \lb
\tr[tr3] 1 \lb 2 \lb 3 \lb
\tr[tr4] 1 \lbb 2 \lbb 3 \lbb
\end{tikzpicture} &
$n$ & 
$n^2(1-p)$ &
$p^5+np^7+n^2p^9$& 
Lem.~\ref{lem9}\\
\begin{tikzpicture}
\path ( 90:1) node[vertex,label=above:1] (1) {} to 
      (210:1) node[vertex,label=left:2 ] (2) {} to 
      (330:1) node[vertex,label=right:3] (3) {} to
      (150:2) node[vertex,label=above:4] (4) {};
\tr[tr1] 1 \rb 2 \rb 3 \rb
\tr[tr2] 1 \rbb 2 \lb 4 \lb
\tr[tr3] 1 \lb 2 \lb 3 \lb
\tr[tr4] 1 \lbb 2 \rb 4 \rb
\end{tikzpicture} &
$n$ & 
$n^2(1-p)$ &
$p^5+np^7+n^2p^9$& 
Lem.~\ref{lem9}\\
\begin{tikzpicture}
\path ( 90:1) node[vertex,label=above:1] (1) {} to 
      (210:1) node[vertex,label=left:2 ] (2) {} to 
      (330:1) node[vertex,label=right:3] (3) {} to
      (150:2) node[vertex,label=above:4] (4) {};
\tr[tr1] 1 \rb 2 \rb 3 \rb
\tr[tr2] 1 \rbb 2 \lb 4 \lb
\tr[tr3] 1 \lb 2 \lb 3 \lb
\tr[tr4] 1 \rbb 3 \lb 4 \rb
\end{tikzpicture} &
$n$& 
$n^2(1-p)$ &
$p^6+~np^8+~n^2p^{10}$ $\leq p^5+np^7+n^2p^{9}$& 
Lem.~\ref{lem9}\\
\begin{tikzpicture}
\path ( 90:1) node[vertex,label=above:1] (1) {} to 
      (210:1) node[vertex,label=left:2 ] (2) {} to 
      (330:1) node[vertex,label=right:3] (3) {} to
      (150:2) node[vertex,label=above:4] (4) {} to
      ( 30:2) node[vertex,label=above:5] (5) {};
\tr[tr1] 1 \rb 2 \rb 3 \rb
\tr[tr2] 1 \rbb 2 \lb 4 \lb
\tr[tr3] 1 \lb 2 \lb 3 \lb
\tr[tr4] 1 \lbb 2 \rb 5 \rb
\end{tikzpicture} &
$n^2$ & 
$n^2(1-p)$ &
$p^7+np^9+n^2p^{11}$& 
Lem.~\ref{lem9}\\
\begin{tikzpicture}
\path ( 90:1) node[vertex,label=above:1] (1) {} to 
      (210:1) node[vertex,label=left:2 ] (2) {} to 
      (330:1) node[vertex,label=right:3] (3) {} to
      (150:2) node[vertex,label=above:4] (4) {} to
      ( 30:2) node[vertex,label=above:5] (5) {};
\tr[tr1] 1 \rb 2 \rb 3 \rb
\tr[tr2] 1 \rbb 2 \lb 4 \lb
\tr[tr3] 1 \lb 2 \lb 3 \lb
\tr[tr4] 1 \lbb 3 \rb 5 \rb
\end{tikzpicture} &
$n^2$ & 
$n^2(1-p)$ &
$p^5+np^7+n^2p^9$& 
Lem.~\ref{lem9}\\
\end{longtable}
\end{center}


\begin{center}\setlength{\tabcolsep}{.5em}
\footnotesize
\begin{longtable}{
	>{\centering}m{0.8in}
	>{\centering}m{0.5in}
	>{\centering}m{0.65in}
	>{\centering}m{1.0in}
	>{\centering\arraybackslash}m{0.5in}
}
\caption{\label{tab2} All possible combinations of triangle indicators that can occur when expanding~$r_{4,1,2}$ and~$r_{4,2,2}$, along with the bounds on the  corresponding summands.}\\
\toprule
\bf Pattern & 
\bf occur\-rences & 
\bf bound for large~$\boldsymbol{p}$&
\bf bound for small~$\boldsymbol{p}$&
\bf Lemma used\\
\midrule
\endfirsthead

\multicolumn{4}{c}%
{\tablename~\thetable\ (continued from previous page)} \\
\midrule
\bf Pattern & 
\bf occur\-rences & 
\bf bound for large~$\boldsymbol{p}$&
\bf bound for small~$\boldsymbol{p}$&
\bf Lemma used\\
\midrule
\endhead

\multicolumn{5}{r}{(continued on next page)}\\
\endfoot

\bottomrule
\endlastfoot

\begin{tikzpicture}
\path (  0:0    ) node[vertex,label=above:1] (1) {} to 
      (270:1.732) node[vertex,label=below:2] (2) {} to 
      (210:1.732) node[vertex,label= left:3] (3) {} to
      (330:1.732) node[vertex,label=right:4] (4) {};
\tr[tr1] 1 \rb 2 \db 3 \db
\tr[tr2] 1 \rbb 2 \lbhb 3 \lbhb
\tr[tr3] 1 \lb 2 \db 4 \db
\tr[tr4] 1 \lbb 2 \rbhb 4 \rbhb
\end{tikzpicture} &
$1$&
$n^2(1-p)$&
$p^5+np^7+n^2p^9$&
Lem.~\ref{lem9} \\
\begin{tikzpicture}
\path (  0:0    ) node[vertex,label=above:1] (1) {} to 
      (270:1.732) node[vertex,label=below:2] (2) {} to 
      (210:1.732) node[vertex,label= left:3] (3) {} to
      (330:1.732) node[vertex,label=right:4] (4) {};
\tr[tr1] 1 \rb 2 \db 3 \db
\tr[tr2] 1 \rbb 2 \lbhb 4 \lbhb
\tr[tr3] 1 \lb 2 \db 4 \db
\tr[tr4] 1 \lbb 2 \rbhb 4 \rbhb
\end{tikzpicture} &
$1$&
$n^2(1-p)$&
$p^5+np^7+n^2p^9$&
Lem.~\ref{lem9} \\
\begin{tikzpicture}
\path (  0:0    ) node[vertex,label=above:1] (1) {} to 
      (270:1.732) node[vertex,label=below:2] (2) {} to 
      (210:1.732) node[vertex,label= left:3] (3) {} to
      (330:1.732) node[vertex,label=right:4] (4) {};
\tr[tr1] 1 \rb 2 \db 3 \db
\tr[tr2] 1 \rbhb 3 \rhb 4 \lbhb
\tr[tr3] 1 \lb 2 \db 4 \db
\tr[tr4] 1 \lbb 2 \rbhb 4 \rbhb
\end{tikzpicture} &
$1$&
$n^2(1-p)$&
$p^6+~np^8+~n^2p^{10}$ $\leq p^5+np^7+n^2p^{9}$& 
Lem.~\ref{lem9} \\
\begin{tikzpicture}
\path (  0:0    ) node[vertex,label=above:1] (1) {} to 
      (270:1.732) node[vertex,label=below:2] (2) {} to 
      (210:1.732) node[vertex,label= left:3] (3) {} to
      (330:1.732) node[vertex,label=right:4] (4) {} to
      (150:1.732) node[vertex,label= left:5] (5) {};
\tr[tr1] 1 \rb 2 \db 3 \db
\tr[tr2] 1 \rbb 2 \lhb 5 \lb
\tr[tr3] 1 \lb 2 \db 4 \db
\tr[tr4] 1 \lbb 2 \rbhb 4 \rbhb
\end{tikzpicture} &
$n$&
$n^2(1-p)$&
$p^7+np^9+n^2p^{11}$&
Lem.~\ref{lem9} \\
\begin{tikzpicture}
\path (  0:0    ) node[vertex,label=above:1] (1) {} to 
      (270:1.732) node[vertex,label=below:2] (2) {} to 
      (210:1.732) node[vertex,label= left:3] (3) {} to
      (330:1.732) node[vertex,label=right:4] (4) {} to
      (150:1.732) node[vertex,label= left:5] (5) {};
\tr[tr1] 1 \rb 2 \db 3 \db
\tr[tr2] 1 \rbhb 3 \lb 5 \lb
\tr[tr3] 1 \lb 2 \db 4 \db
\tr[tr4] 1 \lbb 2 \rbhb 4 \rbhb
\end{tikzpicture} &
$n$&
$n^2(1-p)$&
$p^7+np^9+n^2p^{11}$&
Lem.~\ref{lem9} \\
\begin{tikzpicture}
\path (  0:0    ) node[vertex,label=above:1] (1) {} to 
      (270:1.732) node[vertex,label=below:2] (2) {} to 
      (210:1.732) node[vertex,label= left:3] (3) {} to
      (330:1.732) node[vertex,label=right:4] (4) {};
\tr[tr1] 1 \rb 2 \db 3 \db
\tr[tr2] 1 \rbb 2 \lbhb 4 \lbhb
\tr[tr3] 1 \lb 2 \db 4 \db
\tr[tr4] 1 \lbb 2 \rbhb 3 \rbhb
\end{tikzpicture} &
$1$&
$n^2(1-p)$&
$p^5+np^7+n^2p^9$&
Lem.~\ref{lem9} \\
\begin{tikzpicture}
\path (  0:0    ) node[vertex,label=above:1] (1) {} to 
      (270:1.732) node[vertex,label=below:2] (2) {} to 
      (210:1.732) node[vertex,label= left:3] (3) {} to
      (330:1.732) node[vertex,label=right:4] (4) {};
\tr[tr1] 1 \rb 2 \db 3 \db
\tr[tr2] 1 \rbb 2 \lbhb 4 \lbhb
\tr[tr3] 1 \lb 2 \db 4 \db
\tr[tr4] 1 \rbhb 3 \rhb 4 \rbhb
\end{tikzpicture} &
$1$&
$n^2(1-p)$&
$p^6+~np^8+~n^2p^{10}$ $\leq p^5+np^7+n^2p^{9}$& 
Lem.~\ref{lem9} \\
\begin{tikzpicture}
\path (  0:0    ) node[vertex,label=above:1] (1) {} to 
      (270:1.732) node[vertex,label=below:2] (2) {} to 
      (210:1.732) node[vertex,label= left:3] (3) {} to
      (330:1.732) node[vertex,label=right:4] (4) {} to
      (150:1.732) node[vertex,label= left:5] (5) {};
\tr[tr1] 1 \rb 2 \db 3 \db
\tr[tr2] 1 \rbb 2 \lbhb 4 \lbhb
\tr[tr3] 1 \lb 2 \db 4 \db
\tr[tr4] 1 \lbb 2 \lhb 5 \lb
\end{tikzpicture} &
$n$&
$n^2(1-p)$&
$p^7+np^9+n^2p^{11}$&
Lem.~\ref{lem9} \\
\begin{tikzpicture}
\path (  0:0    ) node[vertex,label=above:1] (1) {} to 
      (270:1.732) node[vertex,label=below:2] (2) {} to 
      (210:1.732) node[vertex,label= left:3] (3) {} to
      (330:1.732) node[vertex,label=right:4] (4) {} to
      (150:1.732) node[vertex,label= left:5] (5) {};
\tr[tr1] 1 \rb 2 \db 3 \db
\tr[tr2] 1 \rbb 2 \lbhb 4 \lbhb
\tr[tr3] 1 \lb 2 \db 4 \db
\tr[tr4] 1 \lbhb 4 \rbbb 5 \lb
\end{tikzpicture} &
$n$&
$n^2(1-p)$&
$p^7+np^9+n^2p^{11}$&
Lem.~\ref{lem9} \\
\begin{tikzpicture}
\path (  0:0    ) node[vertex,label=above:1] (1) {} to 
      (270:1.732) node[vertex,label=below:2] (2) {} to 
      (210:1.732) node[vertex,label= left:3] (3) {} to
      (330:1.732) node[vertex,label=right:4] (4) {};
\tr[tr1] 1 \rb 2 \db 3 \db
\tr[tr2] 1 \rbhb 3 \rhb 4 \rbhb
\tr[tr3] 1 \lb 2 \db 4 \db
\tr[tr4] 1 \lbhb 3 \lhb 4 \lbhb
\end{tikzpicture} &
$1$&
$n^2(1-p)$&
$p^6+~np^8+~n^2p^{10}$ $\leq p^5+np^7+n^2p^{9}$& 
Lem.~\ref{lem9} \\
\begin{tikzpicture}
\path (  0:0    ) node[vertex,label=above:1] (1) {} to 
      (270:1.732) node[vertex,label=below:2] (2) {} to 
      (210:1.732) node[vertex,label= left:3] (3) {} to
      (330:1.732) node[vertex,label=right:4] (4) {};
\tr[tr1] 1 \rb 2 \db 3 \db
\tr[tr2] 1 \rbhb 3 \rhb 4 \rbhb
\tr[tr3] 1 \lb 2 \db 4 \db
\tr[tr4] 2 \lbhb 3 \lhb 4 \lbhb
\end{tikzpicture} &
$1$&
$n^2(1-p)$&
$p^6+~np^8+~n^2p^{10}$ $\leq p^5+np^7+n^2p^{9}$& 
Lem.~\ref{lem9} \\
\begin{tikzpicture}
\path (  0:0    ) node[vertex,label=above:1] (1) {} to 
      (270:1.732) node[vertex,label=below:2] (2) {} to 
      (210:1.732) node[vertex,label= left:3] (3) {} to
      (330:1.732) node[vertex,label=right:4] (4) {} to
      (150:1.732) node[vertex,label= left:5] (5) {};
\tr[tr1] 1 \rb 2 \db 3 \db
\tr[tr2] 1 \rbhb 3 \rhb 4 \rbhb
\tr[tr3] 1 \lb 2 \db 4 \db
\tr[tr4] 1 \lbb 2 \lhb 5 \lb
\end{tikzpicture} &
$n$&
$n^2(1-p)$&
$p^8+~np^{10}+~n^2p^{12}$ $\leq p^7+np^9+n^2p^{11}$& 
Lem.~\ref{lem9} \\
\begin{tikzpicture}
\path (  0:0    ) node[vertex,label=above:1] (1) {} to 
      (270:1.732) node[vertex,label=below:2] (2) {} to 
      (210:1.732) node[vertex,label= left:3] (3) {} to
      (330:1.732) node[vertex,label=right:4] (4) {} to
      (150:1.732) node[vertex,label= left:5] (5) {};
\tr[tr1] 1 \rb 2 \db 3 \db
\tr[tr2] 1 \rbhb 3 \rhb 4 \rbhb
\tr[tr3] 1 \lb 2 \db 4 \db
\tr[tr4] 1 \rbhb 4 \rbbb 5 \lb
\end{tikzpicture} &
$n$&
$n^2(1-p)$&
$p^8+~np^{10}+~n^2p^{12}$ $\leq p^7+np^9+n^2p^{11}$& 
Lem.~\ref{lem9} \\
\begin{tikzpicture}
\path (  0:0    ) node[vertex,label=above:1] (1) {} to 
      (270:1.732) node[vertex,label=below:2] (2) {} to 
      (210:1.732) node[vertex,label= left:3] (3) {} to
      (330:1.732) node[vertex,label=right:4] (4) {} to
      (150:1.732) node[vertex,label= left:5] (5) {};
\tr[tr1] 1 \rb 2 \db 3 \db
\tr[tr2] 1 \rbhb 3 \rhb 4 \rbhb
\tr[tr3] 1 \lb 2 \db 4 \db
\tr[tr4] 2 \rbhb 4 \rbbb 5 \lb
\end{tikzpicture} &
$n$&
$n^2(1-p)$&
$p^8+~np^{10}+~n^2p^{12}$ $\leq p^7+np^9+n^2p^{11}$& 
Lem.~\ref{lem9} \\
\begin{tikzpicture}
\path (  0:0    ) node[vertex,label=above:1] (1) {} to 
      (270:1.732) node[vertex,label=below:2] (2) {} to 
      (210:1.732) node[vertex,label= left:3] (3) {} to
      (330:1.732) node[vertex,label=right:4] (4) {} to
      (150:1.732) node[vertex,label= left:5] (5) {};
\tr[tr1] 1 \rb 2 \db 3 \db
\tr[tr2] 1 \rbb 2 \lhb 5 \lb
\tr[tr3] 1 \lb 2 \db 4 \db
\tr[tr4] 1 \lbb 2 \rhb 5 \rb
\end{tikzpicture} &
$n$&
$n^2(1-p)$&
$p^7+np^9+n^2p^{11}$&
Lem.~\ref{lem9} \\
\begin{tikzpicture}
\path (  0:0    ) node[vertex,label=above:1] (1) {} to 
      (270:1.732) node[vertex,label=below:2] (2) {} to 
      (210:1.732) node[vertex,label= left:3] (3) {} to
      (330:1.732) node[vertex,label=right:4] (4) {} to
      (150:1.732) node[vertex,label= left:5] (5) {} to
      ( 30:1.732) node[vertex,label=right:6] (6) {};
\tr[tr1] 1 \rb 2 \db 3 \db
\tr[tr2] 1 \rbb 2 \lhb 5 \lb
\tr[tr3] 1 \lb 2 \db 4 \db
\tr[tr4] 1 \lbb 2 \rhb 6 \rb
\end{tikzpicture} &
$n^2$&
$n^2(1-p)$&
$p^9+np^{11}+n^2p^{13}$&
Lem.~\ref{lem9} \\
\begin{tikzpicture}
\path (  0:0    ) node[vertex,label=above:1] (1) {} to 
      (270:1.732) node[vertex,label=below:2] (2) {} to 
      (210:1.732) node[vertex,label= left:3] (3) {} to
      (330:1.732) node[vertex,label=right:4] (4) {} to
      (150:1.732) node[vertex,label= left:5] (5) {};
\tr[tr1] 1 \rb 2 \db 3 \db
\tr[tr2] 1 \rbb 2 \lhb 5 \lb
\tr[tr3] 1 \lb 2 \db 4 \db
\tr[tr4] 1 \lbhb 4 \rbbb 5 \rb
\end{tikzpicture} &
$n$&
$n^2(1-p)$&
$p^8+~np^{10}+~n^2p^{12}$ $\leq p^7+np^9+n^2p^{11}$& 
Lem.~\ref{lem9} \\
\begin{tikzpicture}
\path (  0:0    ) node[vertex,label=above:1] (1) {} to 
      (270:1.732) node[vertex,label=below:2] (2) {} to 
      (210:1.732) node[vertex,label= left:3] (3) {} to
      (330:1.732) node[vertex,label=right:4] (4) {} to
      (150:1.732) node[vertex,label= left:5] (5) {} to
      ( 30:1.732) node[vertex,label=right:6] (6) {};
\tr[tr1] 1 \rb 2 \db 3 \db
\tr[tr2] 1 \rbb 2 \lhb 5 \lhb
\tr[tr3] 1 \lb 2 \db 4 \db
\tr[tr4] 1 \lbhb 4 \rhb 6 \rhb
\end{tikzpicture} &
$n^2$&
$n^2(1-p)$&
$p^9+np^{11}+n^2p^{13}$&
Lem.~\ref{lem9} \\
\begin{tikzpicture}
\path (  0:0    ) node[vertex,label=above:1] (1) {} to 
      (270:1.732) node[vertex,label=below:2] (2) {} to 
      (210:1.732) node[vertex,label= left:3] (3) {} to
      (330:1.732) node[vertex,label=right:4] (4) {} to
      (150:1.732) node[vertex,label= left:5] (5) {};
\tr[tr1] 1 \rb 2 \db 3 \db
\tr[tr2] 1 \rbhb 3 \lhb 5 \lhb
\tr[tr3] 1 \lb 2 \db 4 \db
\tr[tr4] 1 \lb 5 \lbbb 4 \rbhb
\end{tikzpicture} &
$n$&
$n^2(1-p)$&
$p^7+np^9+n^2p^{11}$&
Lem.~\ref{lem9} \\
\begin{tikzpicture}
\path (  0:0    ) node[vertex,label=above:1] (1) {} to 
      (270:1.732) node[vertex,label=below:2] (2) {} to 
      (210:1.732) node[vertex,label= left:3] (3) {} to
      (330:1.732) node[vertex,label=right:4] (4) {} to
      (150:1.732) node[vertex,label= left:5] (5) {} to
      ( 30:1.732) node[vertex,label=right:6] (6) {};
\tr[tr1] 1 \rb 2 \db 3 \db
\tr[tr2] 1 \rbhb 3 \lhb 5 \lhb
\tr[tr3] 1 \lb 2 \db 4 \db
\tr[tr4] 1 \lhb 6 \lhb 4 \rbhb
\end{tikzpicture} &
$n^2$&
$n^2(1-p)$&
$p^9+np^{11}+n^2p^{13}$&
Lem.~\ref{lem9} \\
\begin{tikzpicture}
\path (  0:0    ) node[vertex,label=above:1] (1) {} to 
      (270:1.732) node[vertex,label=below:2] (2) {} to 
      (210:1.732) node[vertex,label= left:3] (3) {} to
      (330:1.732) node[vertex,label=right:4] (4) {} to
      (150:1.732) node[vertex,label= left:5] (5) {};
\tr[tr1] 1 \rb 2 \db 3 \db
\tr[tr2] 1 \rbhb 3 \lhb 5 \lhb
\tr[tr3] 1 \lb 2 \db 4 \db
\tr[tr4] 2 \lhb 5 \lbbb 4 \lbhb
\end{tikzpicture} &
$n$&
$n^2(1-p)$&
$p^9+~np^{11}+~n^2p^{13}$ $\leq p^7+np^9+n^2p^{11}$& 
Lem.~\ref{lem9} \\
\begin{tikzpicture}
\path (  0:0    ) node[vertex,label=above:1] (1) {} to 
      (270:1.732) node[vertex,label=below:2] (2) {} to 
      (210:1.732) node[vertex,label= left:3] (3) {} to
      (330:1.732) node[vertex,label=right:4] (4) {} to
      (150:1.732) node[vertex,label= left:5] (5) {} to
      (300:3    ) node[vertex,label=right:6] (6) {};
\tr[tr1] 1 \rb 2 \db 3 \db
\tr[tr2] 1 \rbhb 3 \lhb 5 \lhb
\tr[tr3] 1 \lb 2 \db 4 \db
\tr[tr4] 2 \rhb 6 \rhb 4 \lbhb
\end{tikzpicture} &
$n^2$&
$n^2(1-p)$&
$p^9+np^{11}+n^2p^{13}$&
Lem.~\ref{lem9}
\end{longtable}
\end{center}


\begin{center}
\setlength{\tabcolsep}{.5em}
\footnotesize
\begin{longtable}{
	>{\centering}m{1.3in}
	>{\centering}m{0.5in}
	>{\centering}m{0.65in}
	>{\centering\arraybackslash}m{1.05in}
	>{\centering\arraybackslash}m{0.5in}
}
\caption{\label{tab3} All possible combinations of triangle indicators that can occur when expanding~$r_{4,1,3}$ and~$r_{4,2,3}$, along with the bounds on the  corresponding summands.}\\
\toprule
\bf Pattern & 
\bf occur\-rences & 
\bf bound for large~$\boldsymbol{p}$&
\bf bound for small~$\boldsymbol{p}$&
\bf Lemma used\\
\midrule
\endfirsthead

\multicolumn{4}{c}%
{\tablename~\thetable\ (continued from previous page)} \\
\midrule
\bf Pattern & 
\bf occur\-rences & 
\bf bound for large~$\boldsymbol{p}$&
\bf bound for small~$\boldsymbol{p}$&
\bf Lemma used\\
\midrule
\endhead

\multicolumn{5}{r}{(continued on next page)}\\
\endfoot

\bottomrule
\endlastfoot

\begin{tikzpicture}
\path (  0:0    ) node[vertex,label=below:1] (1) {} to 
      (150:1.732) node[vertex,label= left:2] (2) {} to 
      (210:1.732) node[vertex,label= left:3] (3) {} to
      ( 30:1.732) node[vertex,label=right:4] (4) {} to
      (330:1.732) node[vertex,label=right:5] (5) {};
\tr[tr1] 1 \db 2 \db 3 \db
\tr[tr2] 1 \rbhb 2 \rbhb 3 \rbhb
\tr[tr3] 1 \db 4 \db 5 \db
\tr[tr4] 1 \lbhb 4 \lbhb 5 \lbhb
\end{tikzpicture} &
$1$&
$n(1-p)$&
$p^7+np^9$&
Lem.~\ref{lem10} \\
\begin{tikzpicture}
\path (  0:0    ) node[vertex,label=below:1] (1) {} to 
      (150:1.732) node[vertex,label= left:2] (2) {} to 
      (210:1.732) node[vertex,label= left:3] (3) {} to
      ( 30:1.732) node[vertex,label=right:4] (4) {} to
      (330:1.732) node[vertex,label=right:5] (5) {};
\tr[tr1] 1 \db 2 \db 3 \db
\tr[tr2] 1 \rbhb 2 \lbb 4 \lbhb
\tr[tr3] 1 \db 4 \db 5 \db
\tr[tr4] 1 \lbhb 4 \lbhb 5 \lbhb
\end{tikzpicture} &
$1$&
$n^2(1-p)$&
$p^7+np^9+n^2p^{11}$&
Lem.~\ref{lem9} \\
\begin{tikzpicture}
\path (  0:0    ) node[vertex,label=below:1] (1) {} to 
      (150:1.732) node[vertex,label= left:2] (2) {} to 
      (210:1.732) node[vertex,label= left:3] (3) {} to
      ( 30:1.732) node[vertex,label=right:4] (4) {} to
      (330:1.732) node[vertex,label=right:5] (5) {};
\tr[tr1] 1 \db 2 \db 3 \db
\tr[tr2] 3 \rbhb 2 \lbb 4 \rbb
\tr[tr3] 1 \db 4 \db 5 \db
\tr[tr4] 1 \lbhb 4 \lbhb 5 \lbhb
\end{tikzpicture} &
$1$&
$n^2(1-p)$&
$p^7+np^9+n^2p^{11}$&
Lem.~\ref{lem9} \\
\begin{tikzpicture}
\path (  0:0    ) node[vertex,label=below:1] (1) {} to 
      (150:1.732) node[vertex,label= left:2] (2) {} to 
      (210:1.732) node[vertex,label= left:3] (3) {} to
      ( 30:1.732) node[vertex,label=right:4] (4) {} to
      (330:1.732) node[vertex,label=right:5] (5) {} to
      ( 90:1.732) node[vertex,label=right:6] (6) {};
\tr[tr1] 1 \db 2 \db 3 \db
\tr[tr2] 1 \rbhb 2 \lhb 6 \lhb
\tr[tr3] 1 \db 4 \db 5 \db
\tr[tr4] 1 \lbhb 4 \lbhb 5 \lbhb
\end{tikzpicture} &
$n$&
$n(1-p)$&
$p^9+np^{11}$&
Lem.~\ref{lem10} \\
\begin{tikzpicture}
\path (  0:0    ) node[vertex,label=below:1] (1) {} to 
      (150:1.732) node[vertex,label=north west:2] (2) {} to 
      (210:1.732) node[vertex,label=south west:3] (3) {} to
      ( 30:1.732) node[vertex,label=right:4] (4) {} to
      (330:1.732) node[vertex,label=right:5] (5) {} to
      (180:3    ) node[vertex,label= left:6] (6) {};
\tr[tr1] 1 \db 2 \db 3 \db
\tr[tr2] 3 \lbhb 2 \rhb 6 \rhb
\tr[tr3] 1 \db 4 \db 5 \db
\tr[tr4] 1 \lbhb 4 \lbhb 5 \lbhb
\end{tikzpicture} &
$n$&
$n(1-p)$&
$p^9+np^{11}$&
Lem.~\ref{lem10} \\
\begin{tikzpicture}
\path (  0:0    ) node[vertex,label=below:1] (1) {} to 
      (150:1.732) node[vertex,label= left:2] (2) {} to 
      (210:1.732) node[vertex,label= left:3] (3) {} to
      ( 30:1.732) node[vertex,label=right:4] (4) {} to
      (330:1.732) node[vertex,label=right:5] (5) {};
\tr[tr1] 1 \db 2 \db 3 \db
\tr[tr2] 1 \rbhb 2 \lbb 4 \rbhb
\tr[tr3] 1 \db 4 \db 5 \db
\tr[tr4] 1 \rbhb 4 \rbbb 2 \rbhb
\end{tikzpicture} &
$1$&
$n^2(1-p)$&
$p^7+np^9+n^2p^{11}$&
Lem.~\ref{lem9} \\
\begin{tikzpicture}
\path (  0:0    ) node[vertex,label=below:1] (1) {} to 
      (150:1.732) node[vertex,label= left:2] (2) {} to 
      (210:1.732) node[vertex,label= left:3] (3) {} to
      ( 30:1.732) node[vertex,label=right:4] (4) {} to
      (330:1.732) node[vertex,label=right:5] (5) {};
\tr[tr1] 1 \db 2 \db 3 \db
\tr[tr2] 1 \rbhb 2 \lbb 4 \rbhb
\tr[tr3] 1 \db 4 \db 5 \db
\tr[tr4] 1 \rbhb 5 \lbbbb 2 \rbhb
\end{tikzpicture} &
$1$&
$n^2(1-p)$&
$p^8+~np^{10}+~n^2p^{12}$ $\leq p^7+np^9+n^2p^{11}$& 
Lem.~\ref{lem9} \\
\begin{tikzpicture}
\path (  0:0    ) node[vertex,label=below:1] (1) {} to 
      (150:1.732) node[vertex,label= left:2] (2) {} to 
      (210:1.732) node[vertex,label= left:3] (3) {} to
      ( 30:1.732) node[vertex,label=right:4] (4) {} to
      (330:1.732) node[vertex,label=right:5] (5) {};
\tr[tr1] 1 \db 2 \db 3 \db
\tr[tr2] 1 \rbhb 2 \lbhb 4 \rbhb
\tr[tr3] 1 \db 4 \db 5 \db
\tr[tr4] 2 \lbbb 4 \lbhb 5 \lbbbb
\end{tikzpicture} &
$1$&
$n^2(1-p)$&
$p^8+~np^{10}+~n^2p^{12}$ $\leq p^7+np^9+n^2p^{11}$& 
Lem.~\ref{lem9} \\
\begin{tikzpicture}
\path (  0:0    ) node[vertex,label=below:1] (1) {} to 
      (150:1.732) node[vertex,label= left:2] (2) {} to 
      (210:1.732) node[vertex,label= left:3] (3) {} to
      ( 30:1.732) node[vertex,label=right:4] (4) {} to
      (330:1.732) node[vertex,label=right:5] (5) {} to
      (  0:3    ) node[vertex,label=above:6] (6) {};
\tr[tr1] 1 \db 2 \db 3 \db
\tr[tr2] 1 \rbhb 2 \lbhb 4 \rbhb
\tr[tr3] 1 \db 4 \db 5 \db
\tr[tr4] 1 \rbhb 4 \rbhb 6 \lbhb
\end{tikzpicture} &
$n$&
$n^2(1-p)$&
$p^9+np^{11}+n^2p^{13}$&
Lem.~\ref{lem9} \\
\begin{tikzpicture}
\path (  0:0    ) node[vertex,label=above:1] (1) {} to 
      (150:1.732) node[vertex,label= left:2] (2) {} to 
      (210:1.732) node[vertex,label= left:3] (3) {} to
      ( 30:1.732) node[vertex,label=right:4] (4) {} to
      (330:1.732) node[vertex,label=right:5] (5) {} to
      (270:1.732) node[vertex,label=below:6] (6) {};
\tr[tr1] 1 \db 2 \db 3 \db
\tr[tr2] 1 \rbhb 2 \lbhb 4 \rbhb
\tr[tr3] 1 \db 4 \db 5 \db
\tr[tr4] 1 \rbhb 5 \lhb 6 \lhb
\end{tikzpicture} &
$n$&
$n^2(1-p)$&
$p^9+np^{11}+n^2p^{13}$&
Lem.~\ref{lem9} \\
\begin{tikzpicture}
\path (  0:0    ) node[vertex,label=below:1] (1) {} to 
      (150:1.732) node[vertex,label= left:2] (2) {} to 
      (210:1.732) node[vertex,label= left:3] (3) {} to
      ( 30:1.732) node[vertex,label=north east:4] (4) {} to
      (330:1.732) node[vertex,label=south east:5] (5) {} to
      (  0:3    ) node[vertex,label=right:6] (6) {};
\tr[tr1] 1 \db 2 \db 3 \db
\tr[tr2] 1 \rbhb 2 \lbhb 4 \rbhb
\tr[tr3] 1 \db 4 \db 5 \db
\tr[tr4] 6 \rhb 4 \lbhb 5 \rhb
\end{tikzpicture} &
$n$&
$n(1-p)$&
$p^9+np^{11}+n^2p^{13}$&
Lem.~\ref{lem9} \\
\begin{tikzpicture}
\path (  0:0    ) node[vertex,label=below:1] (1) {} to 
      (150:1.732) node[vertex,label= left:2] (2) {} to 
      (210:1.732) node[vertex,label= left:3] (3) {} to
      ( 30:1.732) node[vertex,label=right:4] (4) {} to
      (330:1.732) node[vertex,label=right:5] (5) {};
\tr[tr1] 1 \db 2 \db 3 \db
\tr[tr2] 3 \lbhb 2 \lbb 4 \rbbb
\tr[tr3] 1 \db 4 \db 5 \db
\tr[tr4] 3 \lbb 4 \lbhb 5 \lbb
\end{tikzpicture} &
$1$&
$n^2(1-p)$&
$p^9+~np^{11}+~n^2p^{13}$ $\leq p^7+np^9+n^2p^{11}$& 
Lem.~\ref{lem9} \\
\begin{tikzpicture}
\path (  0:0    ) node[vertex,label=below:1] (1) {} to 
      (150:1.732) node[vertex,label= left:2] (2) {} to 
      (210:1.732) node[vertex,label= left:3] (3) {} to
      ( 30:1.732) node[vertex,label=right:4] (4) {} to
      (330:1.732) node[vertex,label=right:5] (5) {} to
      (  0:3    ) node[vertex,label=above:6] (6) {};
\tr[tr1] 1 \db 2 \db 3 \db
\tr[tr2] 3 \lbhb 2 \lbb 4 \rbb
\tr[tr3] 1 \db 4 \db 5 \db
\tr[tr4] 1 \rbhb 4 \rbhb 6 \lbhb
\end{tikzpicture} &
$n$&
$n^2(1-p)$&
$p^{10}+~np^{12}+~n^2p^{14}$ $\leq p^9+np^{11}+n^2p^{13}$& 
Lem.~\ref{lem9} \\
\begin{tikzpicture}
\path (  0:0    ) node[vertex,label=right:1] (1) {} to 
      (150:1.732) node[vertex,label= left:2] (2) {} to 
      (210:1.732) node[vertex,label= left:3] (3) {} to
      ( 30:1.732) node[vertex,label=right:4] (4) {} to
      (330:1.732) node[vertex,label=right:5] (5) {} to
      (270:1.732) node[vertex,label=below:6] (6) {};
\tr[tr1] 1 \db 2 \db 3 \db
\tr[tr2] 3 \lbhb 2 \lbb 4 \rbb
\tr[tr3] 1 \db 4 \db 5 \db
\tr[tr4] 1 \rbhb 5 \lhb 6 \lhb
\end{tikzpicture} &
$n$&
$n^2(1-p)$&
$p^{10}+~np^{12}+~n^2p^{14}$ $\kern -0.1em\leq p^9+np^{11}+n^2p^{13}$& 
Lem.~\ref{lem9} \\
\begin{tikzpicture}
\path (  0:0    ) node[vertex,label=below:1] (1) {} to 
      (150:1.732) node[vertex,label= left:2] (2) {} to 
      (210:1.732) node[vertex,label= left:3] (3) {} to
      ( 30:1.732) node[vertex,label=north east:4] (4) {} to
      (330:1.732) node[vertex,label=south east:5] (5) {} to
      (  0:3    ) node[vertex,label=above:6] (6) {};
\tr[tr1] 1 \db 2 \db 3 \db
\tr[tr2] 3 \lbhb 2 \lbb 4 \rbb
\tr[tr3] 1 \db 4 \db 5 \db
\tr[tr4] 5 \rbhb 4 \lhb 6 \lhb
\end{tikzpicture} &
$n$&
$n^2(1-p)$&
$p^{10}+~np^{12}+~n^2p^{14}$ $\leq p^9+np^{11}+n^2p^{13}$& 
Lem.~\ref{lem9} \\
\begin{tikzpicture}
\path (  0:0    ) node[vertex,label=below:1] (1) {} to 
      (150:1.732) node[vertex,label= left:2] (2) {} to 
      (210:1.732) node[vertex,label= left:3] (3) {} to
      ( 30:1.732) node[vertex,label=right:4] (4) {} to
      (330:1.732) node[vertex,label=right:5] (5) {} to
      ( 90:1.732) node[vertex,label=right:6] (6) {};
\tr[tr1] 1 \db 2 \db 3 \db
\tr[tr2] 1 \rbhb 2 \lhb 6 \rb
\tr[tr3] 1 \db 4 \db 5 \db
\tr[tr4] 1 \lbhb 4 \rhb 6 \lb
\end{tikzpicture} &
$n$&
$n^2(1-p)$&
$p^9+np^{11}+n^2p^{13}$&
Lem.~\ref{lem9} \\
\begin{tikzpicture}
\path (  0:0    ) node[vertex,label=right:1] (1) {} to 
      (150:1.732) node[vertex,label= left:2] (2) {} to 
      (210:1.732) node[vertex,label= left:3] (3) {} to
      ( 30:1.732) node[vertex,label=right:4] (4) {} to
      (330:1.732) node[vertex,label=right:5] (5) {} to
      ( 90:1.732) node[vertex,label=right:6] (6) {} to
      (270:1.732) node[vertex,label= left:7] (7) {};
\tr[tr1] 1 \db 2 \db 3 \db
\tr[tr2] 1 \rbhb 2 \lhb 6 \lhb
\tr[tr3] 1 \db 4 \db 5 \db
\tr[tr4] 1 \rbhb 5 \lhb 7 \lb
\end{tikzpicture} &
$n^2$&
$n(1-p)$&
$p^{11}+np^{13}$&
Lem.~\ref{lem10} \\
\begin{tikzpicture}
\path (  0:0    ) node[vertex,label=below:1] (1) {} to 
      (150:1.732) node[vertex,label= left:2] (2) {} to 
      (210:1.732) node[vertex,label= left:3] (3) {} to
      ( 30:1.732) node[vertex,label=right:4] (4) {} to
      (330:1.732) node[vertex,label=right:5] (5) {} to
      ( 90:1.732) node[vertex,label=right:6] (6) {};
\tr[tr1] 1 \db 2 \db 3 \db
\tr[tr2] 1 \rbhb 2 \lhb 6 \lhb
\tr[tr3] 1 \db 4 \db 5 \db
\tr[tr4] 5 \rbhb 4 \rhb 6 \rhb
\end{tikzpicture} &
$n$&
$n^2(1-p)$&
$p^{10}+~np^{12}+~n^2p^{14}$ $\leq p^9+np^{11}+n^2p^{13}$& 
Lem.~\ref{lem9} \\
\begin{tikzpicture}
\path (  0:0    ) node[vertex,label=below:1] (1) {} to 
      (150:1.732) node[vertex,label= left:2] (2) {} to 
      (210:1.732) node[vertex,label= left:3] (3) {} to
      ( 30:1.732) node[vertex,label=north east:4] (4) {} to
      (330:1.732) node[vertex,label=south east:5] (5) {} to
      ( 90:1.732) node[vertex,label=right:6] (6) {} to
      (  0:3    ) node[vertex,label=right:7] (7) {};
\tr[tr1] 1 \db 2 \db 3 \db
\tr[tr2] 1 \rbhb 2 \lhb 6 \lhb
\tr[tr3] 1 \db 4 \db 5 \db
\tr[tr4] 4 \lbhb 5 \rhb 7 \rhb
\end{tikzpicture} &
$n^2$&
$n(1-p)$&
$p^{11}+np^{13}$&
Lem.~\ref{lem10} \\
\begin{tikzpicture}
\path (  0:0    ) node[vertex,label=below:1] (1) {} to 
      (150:1.732) node[vertex,label= left:2] (2) {} to 
      (210:1.732) node[vertex,label= left:3] (3) {} to
      ( 30:1.732) node[vertex,label=right:4] (4) {} to
      (330:1.732) node[vertex,label=right:5] (5) {} to
      ( 90:1.732) node[vertex,label=right:6] (6) {};
\tr[tr1] 1 \db 2 \db 3 \db
\tr[tr2] 3 \lbhb 2 \lhb 6 \lhb
\tr[tr3] 1 \db 4 \db 5 \db
\tr[tr4] 5 \rbhb 4 \rhb 6 \rhb
\end{tikzpicture} &
$n$&
$n^2(1-p)$&
$p^{10}+~np^{12}+~n^2p^{14}$ $\leq p^9+np^{11}+n^2p^{13}$& 
Lem.~\ref{lem9} \\
\begin{tikzpicture}
\path (  0:0    ) node[vertex,label=below:1] (1) {} to 
      (150:1.732) node[vertex,label=north west:2] (2) {} to 
      (210:1.732) node[vertex,label=south west:3] (3) {} to
      ( 30:1.732) node[vertex,label=north east:4] (4) {} to
      (330:1.732) node[vertex,label=south east:5] (5) {} to
      (180:3    ) node[vertex,label= left:6] (6) {} to
      (  0:3    ) node[vertex,label=right:7] (7) {};
\tr[tr1] 1 \db 2 \db 3 \db
\tr[tr2] 3 \lbhb 2 \rhb 6 \rhb
\tr[tr3] 1 \db 4 \db 5 \db
\tr[tr4] 4 \lbhb 5 \rhb 7 \rhb
\end{tikzpicture} &
$n^2$&
$n(1-p)$&
$p^{11}+np^{13}$&
Lem.~\ref{lem10} \\
\end{longtable}
\end{center}


\begin{center}\setlength{\tabcolsep}{.5em}
\footnotesize
\begin{longtable}{
	>{\centering}m{1.3in}
	>{\centering}m{0.5in}
	>{\centering}m{0.65in}
	>{\centering}m{1.05in}
	>{\centering\arraybackslash}m{0.5in}
}
\caption{\label{tab4} All possible combinations of triangle indicators that can occur when expanding~$r_{4,1,4}$ and~$r_{4,2,4}$, along with the bounds on the  corresponding summands.}\\
\toprule
\bf Pattern & 
\bf occur\-rences & 
\bf bound for large~$\boldsymbol{p}$&
\bf bound for small~$\boldsymbol{p}$&
\bf Lemma used\\
\midrule
\endfirsthead

\multicolumn{4}{c}%
{\tablename~\thetable\ (continued from previous page)} \\
\midrule
\bf Pattern & 
\bf occur\-rences & 
\bf bound for large~$\boldsymbol{p}$&
\bf bound for small~$\boldsymbol{p}$&
\bf Lemma used\\
\midrule
\endhead

\multicolumn{5}{r}{(continued on next page)}\\
\endfoot

\bottomrule
\endlastfoot

\begin{tikzpicture}
\path ( 90:1) node[vertex,label=above:1] (1) {} to 
      (210:1) node[vertex,label=below:2 ] (2) {} to 
      (330:1) node[vertex,label=below:3] (3) {};
\path (2.598,0.5)
			+(270:1) node[vertex,label=below:4] (4) {} to 
      +(390:1) node[vertex,label=above:5] (5) {} to 
      +(150:1) node[vertex,label=above:6] (6) {};
\tr[tr1] 1 \db 2 \db 3 \db
\tr[tr2] 1 \rbhb 2 \rbhb 3 \rbhb
\tr[tr3] 4 \db 5 \db 6 \db
\tr[tr4] 4 \rbhb 5 \rbhb 6 \rbhb
\end{tikzpicture} &
$1$&
$1-p$&
$p^9$&
Lem.~\ref{lem11} \\
\begin{tikzpicture}
\path ( 90:1) node[vertex,label=above:1] (1) {} to 
      (210:1) node[vertex,label=below:2 ] (2) {} to 
      (330:1) node[vertex,label=below:3] (3) {};
\path (2.598,0.5)
			+(270:1) node[vertex,label=below:4] (4) {} to 
      +(390:1) node[vertex,label=above:5] (5) {} to 
      +(150:1) node[vertex,label=above:6] (6) {};
\tr[tr1] 1 \db 2 \db 3 \db
\tr[tr2] 1 \lhb 6 \lhb 3 \rbhb
\tr[tr3] 4 \db 5 \db 6 \db
\tr[tr4] 4 \rbhb 5 \rbhb 6 \rbhb
\end{tikzpicture} &
$1$&
$n(1-p)$&
$p^9+np^{11}$&
Lem.~\ref{lem12} \\
\begin{tikzpicture}
\path ( 90:1) node[vertex,label=above:1] (1) {} to 
      (210:1) node[vertex,label=below:2 ] (2) {} to 
      (330:1) node[vertex,label=below:3] (3) {};
\path (2.598,0.5)
			+(270:1) node[vertex,label=below:4] (4) {} to 
      +(390:1) node[vertex,label=above:5] (5) {} to 
      +(150:1) node[vertex,label=above:6] (6) {};
\path (-0.866,0.5)
      +(150:1) node[vertex,label=above:7] (7) {};
\tr[tr1] 1 \db 2 \db 3 \db
\tr[tr2] 1 \rhb 7 \rhb 2 \lbhb
\tr[tr3] 4 \db 5 \db 6 \db
\tr[tr4] 4 \rbhb 5 \rbhb 6 \rbhb
\end{tikzpicture} &
$n$&
$1-p$&
$p^{11}$&
Lem.~\ref{lem11} \\
\begin{tikzpicture}
\path ( 90:1) node[vertex,label=above:1] (1) {} to 
      (210:1) node[vertex,label=below:2 ] (2) {} to 
      (330:1) node[vertex,label=below:3] (3) {};
\path (2.598,0.5)
			+(270:1) node[vertex,label=below:4] (4) {} to 
      +(390:1) node[vertex,label=above:5] (5) {} to 
      +(150:1) node[vertex,label=above:6] (6) {};
\tr[tr1] 1 \db 2 \db 3 \db
\tr[tr2] 1 \lb 6 \rb 3 \rbhb
\tr[tr3] 4 \db 5 \db 6 \db
\tr[tr4] 3 \rb 4 \lbhb 6 \lb
\end{tikzpicture} &
$1$&
$n^2(1-p)$&
$p^9+np^{11}+n^2p^{13}$&
Lem.~\ref{lem9} \\
\begin{tikzpicture}
\path ( 90:1) node[vertex,label=above:1] (1) {} to 
      (210:1) node[vertex,label=below:2 ] (2) {} to 
      (330:1) node[vertex,label=below:3] (3) {};
\path (2.598,0.5)
			+(270:1) node[vertex,label=below:4] (4) {} to 
      +(390:1) node[vertex,label=above:5] (5) {} to 
      +(150:1) node[vertex,label=above:6] (6) {};
\tr[tr1] 1 \db 2 \db 3 \db
\tr[tr2] 1 \lb 6 \lhb 3 \rbhb
\tr[tr3] 4 \db 5 \db 6 \db
\tr[tr4] 2 \lbb 4 \lbhb 6 \rb
\end{tikzpicture} &
$1$&
$n^2(1-p)$&
$p^{10}+~np^{12}+~n^2p^{14}$ $\leq p^9+np^{11}+n^2p^{13}$& 
Lem.~\ref{lem9} \\
\begin{tikzpicture}
\path ( 90:1) node[vertex,label=above:1] (1) {} to 
      (210:1) node[vertex,label=below:2 ] (2) {} to 
      (330:1) node[vertex,label=below:3] (3) {};
\path (2.598,0.5)
			+(270:1) node[vertex,label=below:4] (4) {} to 
      +(390:1) node[vertex,label=above:5] (5) {} to 
      +(150:1) node[vertex,label=above:6] (6) {};
\tr[tr1] 1 \db 2 \db 3 \db
\tr[tr2] 1 \lb 6 \lhb 3 \rbhb
\tr[tr3] 4 \db 5 \db 6 \db
\tr[tr4] 3 \lb 4 \lbhb 5 \rb
\end{tikzpicture} &
$1$&
$n^2(1-p)$&
$p^{10}+~np^{12}+~n^2p^{14}$ $\leq p^9+np^{11}+n^2p^{13}$& 
Lem.~\ref{lem9} \\
\begin{tikzpicture}
\path ( 90:1) node[vertex,label=above:1] (1) {} to 
      (210:1) node[vertex,label=below:2 ] (2) {} to 
      (330:1) node[vertex,label=below:3] (3) {};
\path (2.598,0.5)
			+(270:1) node[vertex,label=below:4] (4) {} to 
      +(390:1) node[vertex,label=above:5] (5) {} to 
      +(150:1) node[vertex,label=above:6] (6) {};
\tr[tr1] 1 \db 2 \db 3 \db
\tr[tr2] 1 \lb 6 \lhb 3 \rbhb
\tr[tr3] 4 \db 5 \db 6 \db
\tr[tr4] 2 \lbb 4 \lbhb 5 \rb
\end{tikzpicture} &
$1$&
$n^2(1-p)$&
$p^{10}+~np^{12}+~n^2p^{14}$ $\leq p^9+np^{11}+n^2p^{13}$& 
Lem.~\ref{lem9} \\
\begin{tikzpicture}
\path ( 90:1) node[vertex,label=above:1] (1) {} to 
      (210:1) node[vertex,label=below:2 ] (2) {} to 
      (330:1) node[vertex,label=below:3] (3) {};
\path (2.598,0.5)
			+(270:1) node[vertex,label=below:4] (4) {} to 
      +(390:1) node[vertex,label=above:5] (5) {} to 
      +(150:1) node[vertex,label=above:6] (6) {};
\path (-0.866,0.5)
      +(150:1) node[vertex,label=above:7] (7) {};
\tr[tr1] 1 \db 2 \db 3 \db
\tr[tr2] 1 \rhb 7 \rhb 2 \lbhb
\tr[tr3] 4 \db 5 \db 6 \db
\tr[tr4] 4 \lhb 1 \lb 6 \rbhb
\end{tikzpicture} &
$n$&
$n(1-p)$&
$p^{11}+np^{13}$&
Lem.~\ref{lem12} \\
\begin{tikzpicture}
\path ( 90:1) node[vertex,label=above:1] (1) {} to 
      (210:1) node[vertex,label=below:2 ] (2) {} to 
      (330:1) node[vertex,label=below:3] (3) {};
\path (2.598,0.5)
			+(270:1) node[vertex,label=below:4] (4) {} to 
      +(390:1) node[vertex,label=above:5] (5) {} to 
      +(150:1) node[vertex,label=above:6] (6) {};
\path (-0.866,0.5)
      +(150:1) node[vertex,label=above:7] (7) {};
\tr[tr1] 1 \db 2 \db 3 \db
\tr[tr2] 1 \rhb 7 \rhb 2 \lbhb
\tr[tr3] 4 \db 5 \db 6 \db
\tr[tr4] 4 \lhb 3 \lhb 6 \rbhb
\end{tikzpicture} &
$n$&
$n(1-p)$&
$p^{11}+np^{13}$&
Lem.~\ref{lem12} \\
\begin{tikzpicture}
\path ( 90:1) node[vertex,label=above:1] (1) {} to 
      (210:1) node[vertex,label=below:2 ] (2) {} to 
      (330:1) node[vertex,label=below:3] (3) {};
\path (2.598,0.5)
			+(270:1) node[vertex,label=below:4] (4) {} to 
      +(390:1) node[vertex,label=above:5] (5) {} to 
      +(150:1) node[vertex,label=above:6] (6) {};
\path (-0.866,0.5)
      +(150:1) node[vertex,label=above:7] (7) {};
\tr[tr1] 1 \db 2 \db 3 \db
\tr[tr2] 1 \rhb 7 \rhb 2 \lbhb
\tr[tr3] 4 \db 5 \db 6 \db
\tr[tr4] 4 \lb 7 \rbhb 6 \rbhb
\end{tikzpicture} &
$n$&
$n(1-p)$&
$p^{11}+np^{13}$&
Lem.~\ref{lem12} \\
\begin{tikzpicture}
\path ( 90:1) node[vertex,label=above:1] (1) {} to 
      (210:1) node[vertex,label=below:2 ] (2) {} to 
      (330:1) node[vertex,label=below:3] (3) {};
\path (2.598,0.5)
			+(270:1) node[vertex,label=below:4] (4) {} to 
      +(390:1) node[vertex,label=above:5] (5) {} to 
      +(150:1) node[vertex,label=above:6] (6) {};
\path (-0.866,0.5)
      +(150:1) node[vertex,label=above:7] (7) {};
\path (3.464,0)
      +(330:1) node[vertex,label=below:8] (8) {};
\tr[tr1] 1 \db 2 \db 3 \db
\tr[tr2] 1 \rhb 7 \rhb 2 \lbhb
\tr[tr3] 4 \db 5 \db 6 \db
\tr[tr4] 4 \rhb 8 \rhb 5 \lbhb
\end{tikzpicture} &
$n^2$&
$1-p$&
$p^{13}$&
Lem.~\ref{lem11} \\
\end{longtable}
\end{center}



\section*{Acknowledgments}
I thank Alexander Tikhomirov for very helpful discussions. Thanks go to the Department of Mathematics at the University of Z\"urich, where part of this article was written, and in particular to Valentin F\'eray for fruitful discussions and for calling the work of \cite{Krokowski2015} to my attention. I also thank the Institute of Mathematical Sciences, NUS, for supporting the workshop \emph{Workshop on New Directions in Stein's Method} in March~2015.  This work was supported in parts by NUS Research Grand R-155-000-167-112.

\setlength{\bibsep}{0.5ex}
\def\bibfont{\small}


\begin{thebibliography}{12}
\providecommand{\natexlab}[1]{#1}
\providecommand{\url}[1]{\texttt{#1}}
\expandafter\ifx\csname urlstyle\endcsname\relax
  \providecommand{\doi}[1]{doi: #1}\else
  \providecommand{\doi}{doi: \begingroup \urlstyle{rm}\Url}\fi

\bibitem[Barbour et~al.(1989)Barbour, Karo{\'n}ski and
  Ruci{\'n}ski]{Barbour1989}
A.~D. Barbour, M.~Karo{\'n}ski and A.~Ruci{\'n}ski (1989).
\newblock A central limit theorem for decomposable random variables with
  applications to random graphs.
\newblock \emph{J.~Combin. Theory Ser.~B} \textbf{47}, \penalty0 125--145.

\bibitem[Bulinskii(1996)]{Bulinskii1996}
A.~V. Bulinskii (1996).
\newblock Rate of convergence in the central limit theorem for fields of
  associated random variables.
\newblock \emph{Theory of Probability and its Applications} \textbf{40},
  \penalty0 136--144.

\bibitem[Chen and R{\"o}llin(2010)]{Chen2010b}
L.~H.~Y. Chen and A.~R{\"o}llin (2010).
\newblock Stein couplings for normal approximation.
\newblock Available at \url{arxiv.org/abs/1003.6039}.

\bibitem[Esary et~al.(1967)Esary, Proschan and Walkup]{Esary1967}
J.~D. Esary, F.~Proschan and D.~W. Walkup (1967).
\newblock Association of random variables, with applications.
\newblock \emph{The Annals of Mathematical Statistics} \textbf{38}, \penalty0
  1466--1474.

\bibitem[Ibragimov and Linnik(1971)]{Ibragimov1971}
I.~A. Ibragimov and Y.~V. Linnik (1971).
\newblock \emph{Independent and stationary sequences of random variables}.
\newblock Wolters-Noordhoff Publishing Groningen.

\bibitem[{Krokowski} et~al.(2017){Krokowski}, {Reichenbachs} and
  {Th\"ale}]{Krokowski2015}
K.~Krokowski, A.~Reichenbachs and C.~Th\"ale (2017).
\newblock {Discrete Malliavin-Stein method: Berry-Esseen bounds for random
  graphs and percolation}.
\newblock \emph{Ann. Probab.} \textbf{45}, \penalty0 1071--1109.

\bibitem[Newman(1980)]{Newman1980}
C.~M. Newman (1980).
\newblock Normal fluctuations and the {FKG} inequalities.
\newblock \emph{Comm. Math. Phys.} \textbf{74}, \penalty0 119--128.

\bibitem[R{\"o}llin and Ross(2015)]{Rollin2012a}
A.~R{\"o}llin and N.~Ross (2015).
\newblock Local limit theorems via {L}andau-{K}olmogorov inequalities.
\newblock \emph{Bernoulli} \textbf{21}, \penalty0 851--880.

\bibitem[Ross(2011)]{Ross2011}
N.~Ross (2011).
\newblock Fundamentals of {S}tein's method.
\newblock \emph{Probab. Surv.} \textbf{8}, \penalty0 210--293.

\bibitem[Ruci{\'n}ski(1988)]{Rucinski1988}
A.~Ruci{\'n}ski (1988).
\newblock When are small subgraphs of a random graph normally distributed?
\newblock \emph{Probab. Theory Related Fields.} \textbf{78}, \penalty0 1--10.

\bibitem[Stein(1972)]{Stein1972}
C.~Stein (1972).
\newblock A bound for the error in the normal approximation to the distribution
  of a sum of dependent random variables.
\newblock In \emph{Proceedings of the Sixth Berkeley Symposium on Mathematical
  Statistics and Probability (Univ. California, Berkeley, Calif., 1970/1971),
  Vol. II: Probability theory}, pages 583--602, Berkeley, Calif. Univ.
  California Press.

\bibitem[Tikhomirov(1980)]{Tikhomirov1980}
A.~N. Tikhomirov (1980).
\newblock Convergence rate in the central limit theorem for weakly dependent
  random variables.
\newblock \emph{Theory Probab. Appl.} \textbf{25}, \penalty0 790--809.

\end{thebibliography}

\end{document}